\documentclass[11pt,reqno]{amsart}
\usepackage{t1enc}
\usepackage[latin1]{inputenc}
\usepackage{amsmath,latexsym,amssymb,graphicx,dsfont,amsthm,amsfonts}
\usepackage{color}
\usepackage{tikz}
\usepackage{hyperref}
\usetikzlibrary{arrows.meta}
\usepackage{yfonts}
\usepackage{umoline,comment}

\renewcommand{\L}{\mathbf{L}}

\newtheorem{Th}{Theorem}[section]
\newtheorem{Lemma}[Th]{Lemma}
\newtheorem{Cor}[Th]{Corollary}
\newtheorem{Prop}[Th]{Proposition}

\newtheorem{Ex}[Th]{Example}

\newtheorem{Remark}[Th]{Remark}
\newtheorem{Remarks}[Th]{Remarks}

\renewcommand{\P}{\mathbb{P}}

\newcommand{\calI}{\mathcal{I}}

\newcommand{\N}{\mathbb{N}}
\newcommand{\Z}{\mathbb{Z}}

\renewcommand{\k}{\mathbf{k}}
\newcommand{\tR}{\widetilde{\mathbf{R}}}

\newcommand{\W}[1]{\mathbf{W}_{#1}}

\newcommand{\e}[1]{\mathbf{e}_{#1}}

\renewcommand{\t}{\delta}
\renewcommand{\k}{\mathbf{k}}

\numberwithin{equation}{section}

\begin{document}

\title[Range of Random Walks on Free Products of Graphs]{Range of Random Walks on Free Products}

\author{Lorenz A. Gilch}


\address{Lorenz A. Gilch: University of Passau, Innstr. 33, 94032 Passau, Germany}

\email{Lorenz.Gilch@uni-passau.de}
\urladdr{http://www.math.tugraz.at/$\sim$gilch/}
\date{\today}

\maketitle

\begin{abstract}
In this article we consider transient random walks on free products of graphs. We prove that the asymptotic range of these random walks exists and is strictly positive. In particular, we show that the range varies real-analytically in terms of probability measures of constant support. Moreover, we prove a central limit theorem associated with the range of the random walk.

\end{abstract}

\section{Introduction}
\label{sec:introduction} 

Let $V_1,V_2$ be finite or countable sets with at least two elements and with distinguished elements $o_i\in V_i$ for $i\in \{1,2\}$, and suppose we are given transition matrices $P_1$ and $P_2$ on $V_1$ and $V_2$. The free product of the sets $V_1$ and $V_2$ is given by \mbox{$V:=V_1\ast V_2$,} the set of all finite words over the alphabet $(V_1\cup V_2)\setminus \{o_1,o_2\}$ of the form $x_1\dots x_n$ such that no two consecutive letters $x_j, x_{j+1}$ arise from the same $V_i$. 
Consider a transient Markov chain $(X_n)_{n\in\N_0}$ on $V$ starting at the empty word $o$, which arises from a convex combination of the transition matrices $P_i$ on the sets $V_i$. For sake of better visualisation, we may equip $V$ with a graph structure: there is an oriented edge from $x\in V$ to $y\in V$ if and only if the single step transition probability of walking from $x$ to $y$ is strictly positive.
For $n\in\N_0$, denote by $\mathbf{R}_n$ the number of different states which are visited up to time $n$, that is, $\mathbf{R}_n=\bigl|\{X_0,X_1,\dots,X_n\}\bigr|$. We are interested whether the sequence of random variables $\mathbf{R}_n/n$ converges almost surely to some constant and, if so, to calculate this constant and to study its behaviour when varying the parameters of the underlying random walk. If the limit exists, it is called the \textit{(asymptotic) range} of the random walk $(X_n)_{n\in\N_0}$. 
\par
It is well-known that the asymptotic range exists for random walks on groups, which are governed by probability measures on the group elements; this is a direct consequence of Kingman's subadditive ergodic theorem, see \cite{kingman}. In particular, there is a nice formula for the asymptotic range in the group setting:
\begin{equation}\label{equ:group-case-formula}
\lim_{n\to\infty} \frac{\mathbf{R}_n}{n} = 1- \P[\exists n\in \N: X_n=e],
\end{equation}
where $e$ is the group identity; see Guivarc'h \cite{guivarch}. In our case of general free products Kingman's subbadditive ergodic theorem is \textit{not} applicable, since we only have a partial composition law for the elements in $V$ and  since vertex transitivity is, in general, missing. Hence, existence of the asymptotic range is not guaranteed. Moreover, it turns out that the above formula does \textit{not} necessarily hold in the non-group setting. Since the asymptotic range is an important random walk characteristic number like the rate of escape or the asymptotic entropy, studying existence of the asymptotic range in the non-group case deserves its own right. We note that the range is a measure of how much of the underlying graph is explored by the random walk.
\par
Random walks on free products have been studied in large variety. Let me outline some results. Asymptotic behaviour of return probabilities of random walks on free products has been studied in many ways; amongst others, Gerl and Woess \cite{gerl-woess}, Woess \cite{woess3}, Sawyer \cite{sawyer78}, Cartwright and Soardi \cite{cartwright-soardi}, Lalley \cite{lalley93,lalley:04} and Candellero and G. \cite{candellero-gilch}. For free products of finite groups, Mairesse and Math\'eus \cite{mairesse1} computed an explicit formula for the drift and the asymptotic entropy. In G. \cite{gilch:07,gilch:11}  different formulas for the drift and also for the entropy of random walks on free products of graphs have been computed. Shi et al. \cite{sidoravicius:18} studied the spectral radius for random walks on some classes of free products of graphs. Finally, let me note that the importance of free products is due to Stallings' Splitting Theorem (see Stallings \cite{stallings:71}): a finitely generated group $\Gamma$ has more than one (geometric) end if and only if $\Gamma$ admits a non-trivial decomposition as a free product by amalgamation or an HNN extension over a finite subgroup. Both classes of groups are discussed and studied in detail, e.g., in Lyndon and Schupp \cite{lyndon-schupp}. Recall that free products are amalgams over the trivial subgroup. 
\par
While most of the articles mentioned above deal with random walks on free products of groups, which have a very homogeneous structure, this article goes a step beyond  general free products, which have a much less homogeneous structure.
\par
In the following we summarize the main results of this article. First, we are able to guarantee existence of the asymptotic range for random walks on free products under some mild assumptions:
\begin{Th}\label{thm:range-existence}
Assume that the radius of convergence of the Green function $G(o,o|z)$ defined in (\ref{equ:green-functions}) is strictly bigger than $1$. Then there exists some real constant $\mathfrak{r}>0$ such that
$$
\mathfrak{r}=\lim_{n\to\infty} \frac{\mathbf{R}_n}{n}  \quad \textrm{ almost surely}.
$$
\end{Th}
The second main result states a central limit theorem for the sequence $(\mathbf{R}_n)_{n\in\N_0}$, where we use some regeneration process which is  extracted from a careful analysis of the random walk's trajectory to ``infinity''. We denote by $(T_i)_{i\in\N}$ the associated random times at which that process regenerates; see (\ref{equ:def-Ti}) for the formal definition of these regeneration times.
\begin{Th}\label{thm:clt}
Additionally to the assumption in Theorem \ref{thm:range-existence}, assume that there are $x\in V_1$ and $n\in\mathbb{N}$ such that $\mathbb{P}[X_n=x\mid X_0=x]>0$. Then
the asymptotic range $\mathfrak{r}$ satisfies 
$$
\frac{\mathbf{R}_n- n\cdot \mathfrak{r} }{\sqrt{n}}\xrightarrow{\mathcal{D}} N(0,\sigma^2),
$$
where 
$\displaystyle \sigma^2=\frac{\mathbb{E}\bigl[\bigl(\mathbf{R}_{T_1}-\mathbf{R}_{T_0}-(T_1-T_0)\mathfrak{r}\bigr)^2\bigr]}{\mathbb{E}[T_1-T_0]}$.
\end{Th}
The assumption in the theorem above just ensures that the random walk does not visit a new element in each step, which would lead to $\mathbf{R_n}=n$ for all $n\in\mathbb{N}$; see Remarks \ref{remarks2.2}.(2).
\par
The third main result demonstrates that the range varies real-analytically in terms of transition matrices of constant support depending on finitely many parameters. For this purpose, we assume that all strictly positive single-step transition probabilities  $p(x,y)$, $x,y\in V$, satisfy $p(x,y)\in\{p_1,\dots,p_d\}\subset (0,1)$ for some $d\in\N$. Hence, for any $(p_1,\dots,p_d)\in(0,1)^d$ permitting a well-defined random walk on $V$, we study the mapping $(p_1,\dots,p_d)\mapsto \mathfrak{r}=\mathfrak{r}(p_1,\dots,p_d)$.
\begin{Th}\label{thm:analyticity}
Additionally to the assumption in Theorem \ref{thm:range-existence}, assume that the single step transition probabilities of the random walk $(X_n)_{ n\in\N_0}$ take only finitely many non-negative values $p_1,\dots,p_d\in (0,1), d\in\N$. Then the mapping
$$
(p_1,\dots,p_d) \mapsto \mathfrak{r}=\mathfrak{r}(p_1,\dots,p_d) 
$$
varies real-analytically.
\end{Th}
Analyticity means that, for any parameter vector $\underline{p}_0=(p_{0,1},\dots,p_{0,d})\in(0,1)^d$ permitting a well-defined random walk on $V$, we can expand $\mathfrak{r}(p_1,\dots,p_d)$ as a multivariate power series in the variables $p_1,\dots,p_d$ in a neighbourhood of $\underline{p}_0$.
The details are given in \mbox{Section \ref{sec:analyticity}.}
\par
The proofs involve, in a very crucial way, generating function techniques for free products. These techniques for rewriting probability generating functions on the free product in terms of functions on the single factors of the free product were introduced independently and simultaneously by \cite{cartwright-soardi},  \cite{woess3}, Voiculescu \cite{voiculescu} and McLaughlin \cite{mclaughlin}. In particular, in  \cite{gilch:07} heavy use of generating functions was made and that article will also serve as one of the main references.
\par
The plan of this paper is as follows: in Section \ref{sec:free-products} we give an introduction to free products of graphs and introduce a natural class of random walks on them. In Section \ref{sec:existence-range} we prove existence of the asymptotic range for random walks on free products of graphs (Theorem \ref{thm:range-existence}) by tracing the random walk's path to infinity with the help of exit times. In Section \ref{sec:clt} we derive the proposed Central Limit Theorem \ref{thm:clt}, while in Section \ref{sec:analyticity} we prove Theorem \ref{thm:analyticity}. Finally, in Section \ref{sec:remarks} we give some additional remarks.

\section{Free Products and Random Walks}
\label{sec:free-products}

We recall the definition of general free products, introduce a natural class of random walks on them, and provide several basic properties. In particular, we will introduce generating functions which will serve as a main tool for our analysis. We refer to \cite{gilch:07}, where essential prework has been done and will be used throughout this article.

\subsection{Free Products of Graphs}
\label{subsec:free-products}

Let $V_1,V_2$ be  finite or countable sets with $|V_i|\geq 2$ for $i\in\calI:=\{1,2\}$. For $i\in\calI$, we 
select a distinguished element $o_i$ of $V_i$ as the ``root'' of $V_i$. W.l.o.g. we assume that the sets $V_i$ are pairwise disjoint.
On each $V_i$ consider a (time-)homogeneous random walk with transition matrix $P_i=(p_i(x,y))_{x,y\in V_i}$.
The corresponding $n$-step transition probabilities are denoted by $p_i^{(n)}(x, y)$, where $x, y\in V_i$.  For better visualization, we think of graphs $\mathcal{X}_i$
with vertex sets $V_i$ and roots $o_i$ such that there is an oriented edge $x\to y$ if and only if $p_i(x, y) > 0$.
Furthermore, we shall assume that for every $i\in\calI$ and every $x\in V_i$ there is some $n_x\in\N$ such that $p_i^{(n_x)}(o_i,x) > 0$.  For sake of simplicity, we assume $p_i(x, x)= 0$ for every $i\in\calI$ and $x\in V_i$; for further remarks, see Section \ref{sec:remarks}. We shall also assume that there are $i\in\calI$, $x_1,x_2\in V_i$ and $m,n\in\N$ such that $p_i^{(m)}(x,y)>0$ and $p_i^{(n)}(y,x)>0$; we denote this assumption by (A) and refer to  Remarks \ref{remarks2.2}.(2) for further remarks.
\par
For $i\in\calI$, set $V_i^\times:= V_i\setminus\{o_i\}$ and $V^\times_\ast := V_1^\times \cup V_2^\times$. The \textit{free product of $V_1$ and $V_2$} is given by the set
\begin{equation}\label{equ:free-product}
V:=V_1\ast V_2 := \bigl\lbrace x_1x_2\dots x_n \,\bigl|\, n\in\N, x_j\in V^\times_\ast, x_j\in V_k^\times \Rightarrow x_{j+1}\notin V_k^\times\bigr\rbrace \cup\{o\},
\end{equation}
the set of words over the alphabet $V^\times_\ast$ such that no two consecutive letters come from the same $V_i^\times$,
where $o$ describes the empty word. We have a natural partial composition law on $V$: if $u=u_1\dots  u_m\in V$ and $v = v_1\dots v_n \in V$ with $u_m\in V_i^\times$, $i\in\calI$, and $v_1\notin V_i^\times$, then $uv\in V$ stands for their concatenation as words. In particular, we set $uo_i := u$ for all $i\in\calI$ and $o_iu := u$. Note that $V_i^\times\subseteq V$ and $o_i$ as a word in V is identified with $o$. Since concatenation of words is only partially defined, things are getting much more complicated than in the case of free products of groups (see, e.g., \cite{lyndon-schupp} for more details on free products of groups). Throughout this paper we will use the representation in (\ref{equ:free-product}) for elements in $V$.
The \textit{word length} of a word $u = u_1\dots u_m$ is defined as $\Vert u\Vert:= m$. Additionally, we set $\Vert o\Vert := 0$. The type $\t(u)$ of $u$ is defined to be $i\in\calI$ if $u_m\in V_i^\times$; we set $\t(o) := 0$. 
\par
The set $V$ can again be interpreted as the vertex set of a graph $\mathcal{X}$ which is constructed as follows: take copies of $\mathcal{X}_1$ and $\mathcal{X}_2$ and glue them together at their roots to one single common root, which becomes $o$; inductively, at each vertex $v=v_1\dots v_k$ with $v_k\in V_i$ attach a copy of $\mathcal{X}_j$, $j\in\calI \setminus\{i\}$, where $v$ is identified with $o_j$ from the new copy of $\mathcal{X}_j$. Then $\mathcal{X}$ is the \textit{free product of the graphs $\mathcal{X}_1$ and $\mathcal{X}_2$}.
A \textit{path} of length $n\in\N$ in $\mathcal{X}$ is then a sequence of vertices $(z_0,z_1,\dots,z_n)$ in $V$ such that there is an edge from $z_{i-1}$ to $z_i$ for each $i\in\{1,\dots,n\}$. This gives rise to a natural graph distance and, for $x\in V$, we denote by $|x|$ the \textit{length of a shortest path} from $o$ to $x$.

\begin{Ex}\label{ex:free-product}
Consider the sets $V_1=\{o_1,a\}$ and $V_2=\{o_2,b,c\}$ with the following graph structure:
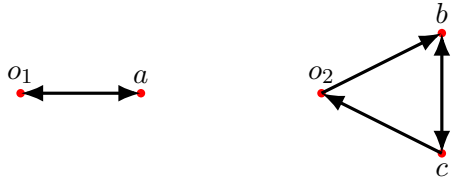
\begin{center}
\begin{tikzpicture}[scale=.8]
\coordinate[label=above:$o_1$] (e) at (0,0);
\coordinate[label=above:$a$] (a) at (2,0);

\coordinate[label=above:$o_2$] (e2) at (5,0);
\coordinate[label=above:$b$] (b) at (7,1);
\coordinate[label=below:$c$] (c) at (7,-1);

\fill[red] (e) circle (2pt);
\fill[red] (a) circle (2pt);
\fill[red] (b) circle (2pt);
\fill[red] (c) circle (2pt);
\fill[red] (e2) circle (2pt);
\draw[{Latex[length=3mm]}-{Latex[length=3mm]},very thick] (e) -- (a);
\draw[-{Latex[length=3mm]},very thick] (e2) -- (b);
\draw[-{Latex[length=3mm]},very thick] (c) -- (e2);
\draw[{Latex[length=3mm]}-{Latex[length=3mm]},very thick] (b) -- (c);

\end{tikzpicture}
\end{center}
\pagebreak[4]
The free product $V_1\ast V_2$ has the the following structure:
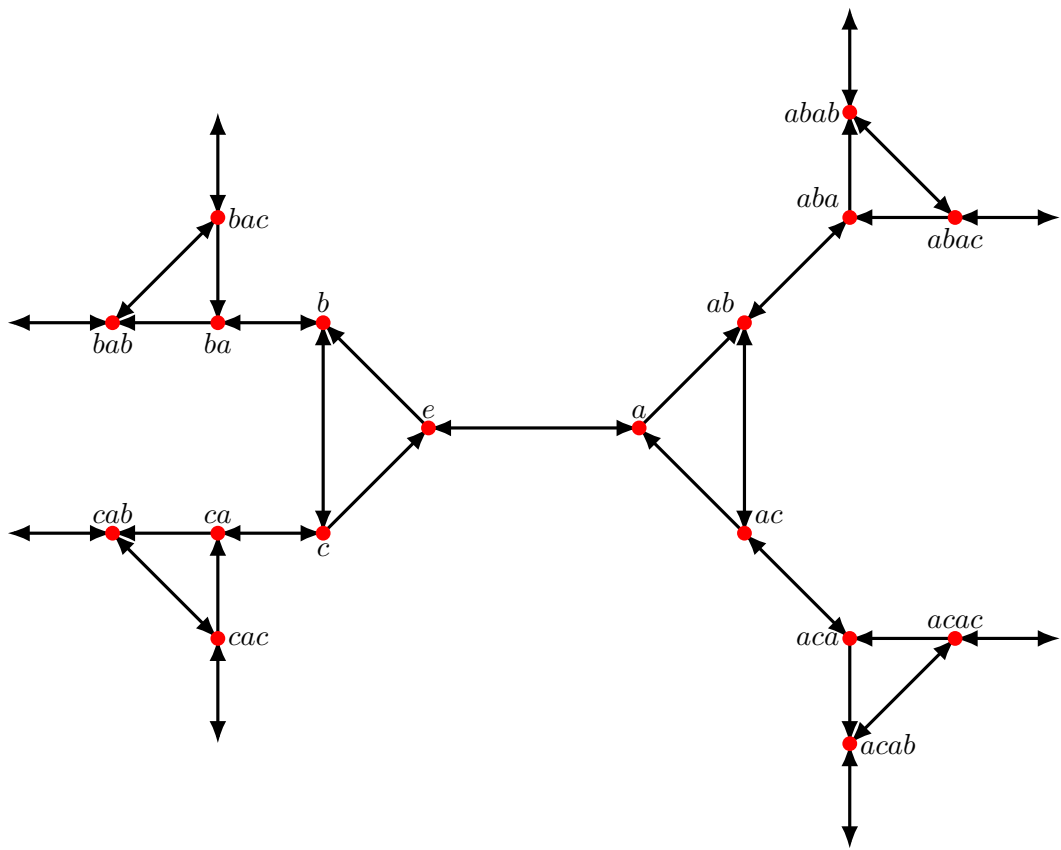
\begin{center}
\begin{tikzpicture}[scale=1.4]
\coordinate[label=above:$e$] (e) at (0,0);
\coordinate[label=above:$a$] (a) at (2,0);
\coordinate[label=170:$ab$] (ab) at (3,1);
\coordinate[label=10:$ac$] (ab2) at (3,-1);
\coordinate[label=170:$aba$] (aba) at (4,2);
\coordinate[label=left:$abab$] (abab) at (4,3);
\coordinate[label=below:$abac$] (abab2) at (5,2);
\coordinate[label=left:$aca$] (ab2a) at (4,-2);
\coordinate[label=right:$acab$] (ab2ab) at (4,-3);
\coordinate[label=above:$acac$] (ab2ab2) at (5,-2);

\coordinate[label=above:$b$] (b) at (-1,1);
\coordinate[label=below:$c$] (b2) at (-1,-1);

\coordinate[label=below:$ba$] (ba) at (-2,1);
\coordinate[label=below:$bab$] (bab) at (-3,1);
\coordinate[label=right:$bac$] (bab2) at (-2,2);

\coordinate[label=above:$ca$] (b2a) at (-2,-1);
\coordinate[label=above:$cab$] (b2ab) at (-3,-1);
\coordinate[label=right:$cac$] (b2ab2) at (-2,-2);

\draw[{Latex[length=3mm]}-{Latex[length=3mm]},very thick] (e) -- (a);
\draw[-{Latex[length=3mm]},very thick] (e) -- (b);
\draw[{Latex[length=3mm]}-{Latex[length=3mm]},very thick] (b) -- (b2);
\draw[-{Latex[length=3mm]},very thick] (b2) -- (e);
\draw[-{Latex[length=3mm]},very thick] (a) -- (ab);
\draw[{Latex[length=3mm]}-{Latex[length=3mm]},very thick] (ab) -- (ab2);
\draw[-{Latex[length=3mm]},very thick] (ab2) -- (a);
\draw[{Latex[length=3mm]}-{Latex[length=3mm]},very thick] (ab) -- (aba);
\draw[-{Latex[length=3mm]},very thick] (aba) -- (abab);
\draw[{Latex[length=3mm]}-{Latex[length=3mm]},very thick] (abab) -- (abab2);
\draw[-{Latex[length=3mm]},very thick] (abab2) -- (aba);
\draw[{Latex[length=3mm]}-{Latex[length=3mm]},very thick] (abab) -- (4,4);
\draw[{Latex[length=3mm]}-{Latex[length=3mm]},very thick] (abab2) -- (6,2);
\draw[{Latex[length=3mm]}-{Latex[length=3mm]},very thick] (ab2) -- (ab2a);
\draw[-{Latex[length=3mm]},very thick] (ab2a) -- (ab2ab);
\draw[{Latex[length=3mm]}-{Latex[length=3mm]},very thick] (ab2ab) -- (ab2ab2);
\draw[-{Latex[length=3mm]},very thick] (ab2ab2) -- (ab2a);
\draw[{Latex[length=3mm]}-{Latex[length=3mm]},very thick] (ab2ab2) -- (6,-2);
\draw[{Latex[length=3mm]}-{Latex[length=3mm]},very thick] (ab2ab) -- (4,-4);
\draw[{Latex[length=3mm]}-{Latex[length=3mm]},very thick] (b) -- (ba);
\draw[-{Latex[length=3mm]},very thick] (ba) -- (bab);
\draw[{Latex[length=3mm]}-{Latex[length=3mm]},very thick] (bab) -- (bab2);
\draw[-{Latex[length=3mm]},very thick] (bab2) -- (ba);
\draw[{Latex[length=3mm]}-{Latex[length=3mm]},very thick] (bab2) --+ (0,1);
\draw[{Latex[length=3mm]}-{Latex[length=3mm]},very thick] (bab) --+ (-1,0);

\draw[{Latex[length=3mm]}-{Latex[length=3mm]},very thick] (b2) -- (b2a);
\draw[-{Latex[length=3mm]},very thick] (b2a) -- (b2ab);
\draw[{Latex[length=3mm]}-{Latex[length=3mm]},very thick] (b2ab) -- (b2ab2);
\draw[-{Latex[length=3mm]},very thick] (b2ab2) -- (b2a);
\draw[{Latex[length=3mm]}-{Latex[length=3mm]},very thick] (b2ab2) --+ (0,-1);
\draw[{Latex[length=3mm]}-{Latex[length=3mm]},very thick] (b2ab) --+ (-1,0);

\fill[red] (e) circle (2pt);
\fill[red] (a) circle (2pt);
\fill[red] (b) circle (2pt);
\fill[red] (b2) circle (2pt);
\fill[red] (ab) circle (2pt);
\fill[red] (ab2) circle (2pt);
\fill[red] (aba) circle (2pt);
\fill[red] (abab) circle (2pt);
\fill[red] (abab2) circle (2pt);
\fill[red] (ab2a) circle (2pt);
\fill[red] (ab2ab) circle (2pt);
\fill[red] (ab2ab2) circle (2pt);
\fill[red] (ba) circle (2pt);
\fill[red] (bab) circle (2pt);
\fill[red] (bab2) circle (2pt);
\fill[red] (b2a) circle (2pt);
\fill[red] (b2ab) circle (2pt);
\fill[red] (b2ab2) circle (2pt);

\end{tikzpicture}
\end{center}
\end{Ex}

The graph structure motivates the following definition.
The \textit{cone} rooted at $x\in V$ is given by the set
$$
C(x):=\bigl\lbrace y\in V \mid y \textrm{ has prefix } x\bigr\rbrace,
$$
that is, $C(x)$ consists of all words $y\in V$ such that each path from $o$ to $y$ in the graph $\mathcal{X}$ has to pass through $x$. In particular, we have $C(o)=V$. 
E.g., in Example \ref{ex:free-product} we have 
$C(ac)=\{ac,aca,acab,acac,\dots\}$.

\subsection{Random Walks on Free Products}
\label{subsec:rw-on-fp}

The next step is the construction of a random walk on the free product arising from $P_1,P_2$ in a natural way. For this purpose, we lift the transition matrices $P_1$ and $P_2$ to transition matrices $\bar P_i=\bigl(\bar p_i(x,y)\bigr)_{x,y\in V}$, $i\in\calI$, \mbox{on $V$:} if $x\in V$ with $\t(x)\neq i$ and $v,w\in V_i$, then $\bar p_i(xv,xw):=p_i(v,w)$. Otherwise, we set $\bar p_i(x,y):=0$. Choose $\alpha\in (0,1)$. Then we define a new transition matrix $P$ on V given by
$$
P=\alpha \cdot \bar P_1 + (1-\alpha)\cdot \bar P_2,
$$
which governs a nearest neighbour random walk on $\mathcal{X}$.
That is, standing at any vertex  $x\in V, \t(x)=i\in\calI$, we first toss a coin (with probability $\alpha$ for ``head'') and decide afterwards to perform one step within the copy of $\mathcal{X}_i$ to which $x$ belongs according to $P_i$ or to perform one step into the new copy of $\mathcal{X}_j$, $j\in\calI\setminus\{j\}$, attached at $x$ according to $P_j$. We assume that $P$ governs a \textit{transient} random walk $(X_n)_{n\in\N_0}$ on $V$ starting at $X_0=o$. For $x, y\in V$, the associated single and $n$-step transition probabilities are denoted by $p(x, y)$ and $p^{(n)}(x, y)$. Thus, $P$ governs a nearest neighbour random walk on the graph $\mathcal{X}$, where $P$ arises from a convex combination of the nearest neighbour random walks on the graphs $\mathcal{X}_1$ and $\mathcal{X}_2$. This definition ensures that every path $(w_0,\dots,w_n)$ in $\mathcal{X}$ has positive probability $\P\bigl[X_1=w_1,\dots,X_n=w_n|X_0=w_0\bigr]>0$.
We use the notation \mbox{$\P_x[\,\cdot\,]:=\P[\,\cdot\, \mid X_0=x]$} for $x\in V$.

In this article we are interested in the number of distinct vertices which are visited by the random walk. In particular, we study the speed at which new vertices are visited. This measures in some sense how much of the graph is explored by the random walk. The number of vertices visited by the random walk $(X_n)_{n\in\N_0}$ until time $n$ is given by 
$$
\mathbf{R}_n:=\bigl|\{X_0,X_1,\dots,X_n\}\bigr|.
$$
If there is a real number $\mathfrak{r}$ such that
$$
\mathfrak{r} = \lim_{n\to\infty} \frac{\mathbf{R}_n}{n} \quad \textrm{almost surely,}
$$
then $\mathfrak{r}$ is called the \textit{(asymptotic) range} of the random walk $(X_n)_{n\in\N_0}$.
 While existence of the asymptotic range is well-known for random walks on groups due to Kingman's subadditive ergodic theorem, existence for non-group random walks is \textit{not} guaranteed a-priori. In particular, since we have no group operation on $V$  (we only have a partial composition law of words)  we cannot apply the reasoning from the group case. This was the starting point for the present article to study existence of the asymptotic range for general free products of graphs, which form an important class of graphs.

Let us mention \cite[Theorem 3.3]{gilch:07}, which demonstrates existence (including a formula) of a positive number $\ell$, the rate of escape w.r.t. the word length (or block length), such that 
$$
\ell = \lim_{n\to\infty} \frac{\Vert X_n\Vert}{n} \quad \textrm{almost surely.}
$$
Denote by $V_\infty$ the set of infinite words $y_1y_2y_3\dots$ over the alphabet $V^\times_\ast$ such that no two consecutive letters arise from the same $V_i^\times$. For $x\in V$ and $y\in V_\infty$, denote by $x\wedge y$ the common prefix of maximal length of $x$ and $y$.
In \cite[Proposition 2.5]{gilch:07} it is shown that the random walk $(X_n)_{n\in\N_0}$ converges to some $V_\infty$-valued random variable $X_\infty$ in the sense that the length of the common prefix of $X_n$ and $X_\infty$ tends to infinity almost surely. In other words, $\lim_{n\to\infty} \Vert X_n\wedge X_\infty\Vert=\infty$ almost surely. 
We will make use of these results in the proofs later.

\begin{Remarks}\label{remarks2.2}
~\par
\begin{enumerate}
\item If $|V_1|=|V_2|=2$ and $p_i(x,y)=1-\mathds{1}_{x}(y)$ for $x,y\in V_i$, $i\in\calI$, then $V$ becomes the free product of groups 
$$
V=\bigl(\Z/(2\Z)\bigr)\ast \bigl(\Z/(2\Z)\bigr).
$$
In this case the underlying random walk is group-invariant and $(X_n)_{n\in\N_0}$ is recurrent. Moreover, existence of the asymptotic range $\mathfrak{r}$ is already known in this case and a formula is given by (\ref{equ:group-case-formula}), that is, $\mathfrak{r}=0$. If at least one out of $P_1$ and $P_2$ is \textit{not} irreducible, then the random walk on $V$ is transient and we may apply the techniques below.
\item In Subsection \ref{subsec:free-products} we made the assumption (A) that there are $i\in\calI$ and $x,y\in V_i$ with $p_i^{(m)}(x,y)>0$ and $p_i^{(n)}(y,x)>0$ for some $m,n\in\N$. If this assumption does \textit{not} hold, then there are no circles in the graph $\mathcal{X}$, that is, at every instant of time a new vertex is visited, yielding $\mathbf{R}_n=n$ and therefore $\mathfrak{r}=1$. Hence, we may exclude this case from now on.
\end{enumerate}
\end{Remarks}

\subsection{Generating Functions}

Our main tool will be the usage of generating functions, which we introduce now. The \textit{Green functions} related to $P_i$ and $P$ are given by
\begin{equation}\label{equ:green-functions}
G_i(x_i, y_i|z):=\sum_{n\geq 0} p_i^{(n)}(x_i, y_i)\cdot z^n \quad \textrm{ and } \quad  
G(x, y|z) :=\sum_{n\geq 0} p^{(n)}(x, y)\cdot z^n,
\end{equation}
where $z\in\mathbb{C}$, $x_i, y_i\in V_i$ and $x,y\in V$.  At this point we make the \textit{basic assumption} that the radii of convergence of $G(\cdot,\cdot |z)$ are at least $R>1$. This implies transience of our random walk on $V$. If the random walk on $V$ is irreducible, then all Green functions have common radius of convergence $R>1$; in the reducible case one can also easily show that the Green functions $G(x,y|z)$, $x,y\in V$, have radii of convergence $R_{x,y}$ with $R_{x,y}\geq R>1$ for some $R>1$. For instance, if $P_1$ and $P_2$ govern irreducible and reversible random walks, then $R>1$; see Woess \cite[Theorem 10.3]{woess}.

The \textit{first visit generating function} related to $P$ is given by
\begin{eqnarray*}
F(x,y|z):= \sum_{n\geq 0} \P_x\bigl[ \forall k<n: X_k\neq y,X_n=y\bigr]\cdot z^n,
\end{eqnarray*}
while the \textit{first return generating function} related to $P$ is given by
\begin{eqnarray*}
U(x,y|z):= \sum_{n\geq 1} \P_x\bigl[ \forall k\in\{1,\dots,n-1\}: X_k\neq y,X_n=y\bigr]\cdot z^n
\end{eqnarray*}
and the \textit{last visit generating function} related to $P$ is given by
\begin{eqnarray*}
L(x,y|z):= \sum_{n\geq 0} \P_x\bigl[ \forall k\in\{1,\dots,n\}: X_k\neq x,X_n=y\bigr]\cdot z^n.
\end{eqnarray*}
Analogously, we write $L_i(\cdot,\cdot|z)$, $i\in\calI$, for the corresponding last visit generating functions associated with the random walk on $V_i$ governed by $P_i$.

Recall the following important equations (e.g., see  \cite[Lemma 1.13]{woess}, \mbox{G. \cite[Lemma 1.6]{gilch}):}
\begin{eqnarray}
G(x,y|z) &=& F(x,y|z)\cdot G(y,y|z),\label{equ:Green1}\\
G(x,y|z) &=& G(x,x|z)\cdot L(x,y|z),\label{equ:Green3}\\
G(x,x|z) & = & \frac{1}{1-U(x,x|z)}.\label{equ:Green2}
\end{eqnarray}
If every path from $x\in V$ to $y\in V$ has to pass through $w\in V$, then
\begin{equation}\label{equ:L-factorisation}
L(x,y|z)=L(x,w|z)\cdot L(w,y|z);
\end{equation}
see \cite[Lemma 1.6]{gilch}.
\par
For $i\in\calI$, define
$$
\xi_i(z):=\sum_{n\geq 1} \P[X_n\in V_i^\times,\forall m<n: X_m\notin V_i^\times]\cdot z^n,
$$
We write $\xi_i:=\xi_i(1)$, the probability of visiting any element in $V_i^\times$ after finite time when starting at $o$. 
Due to the structure of free products it is easy to check that, for all $x\in V\setminus\{o\}$ with $\t(x)=i$,
\begin{equation}\label{equ:xi}
\xi_i=\P\bigl[\exists n\in\N: X_n\notin C(x) \mid X_0=x\bigr];
\end{equation}
in particular, the probability on the left hand side does only depend on the type $\t(x)$ of $x$.
In \cite[Lemma 2.3]{gilch:07} it is shown that the important strict inequality $\xi_i<1$ holds for all $i\in\calI$. Moreover, $\xi_i(z)$ has radius of convergence strictly bigger than $1$ due to
$$
\xi_i(z)= \sum_{n\geq 0}\P\bigl[X_n=o,\forall m<n: X_m\notin V_i^\times\bigr]\cdot z^n \cdot \alpha_i\cdot z \leq
G(o,o|z)\cdot \alpha_i\cdot z \quad \textrm{ for } z>0.
$$
Furthermore, we have the following important equation: if $x_i,y_i\in V_i$, $i\in\calI$, then
\begin{equation}\label{equ:L-Li}
L(x_i,y_i|z)=L_i\bigl(x_i,y_i\,\bigl|\, \xi_i(z)\bigr);
\end{equation}
see \cite[Proposition 2.7]{gilch}. The last equation together with (\ref{equ:Green3}) yields
\begin{equation}\label{equ:L-G}
L(x_i,y_i|z)=L_i\bigl(x_i,y_i\,\bigl|\, \xi_i(z)\bigr)=\frac{G_i\bigl(x_i,y_i\,\bigl|\,\xi_i(z)\bigr)}{G_i\bigl(x_i,x_i\,\bigl|\,\xi_i(z)\bigr)}.
\end{equation}
Finally, we observe that, for $i\in\calI$ and $z\in\mathbb{C}$ with $|z|<R$,
\begin{equation}\label{equ:Gi-sum}
\sum_{g\in V_i} G_i\bigl(o_i,g\,\bigl|\,\xi_i(z)\bigr) = \sum_{n\in\N_0}\sum_{g\in V_i} p_i^{(n)}(o_i,g)\cdot \xi_i(z)^n = \frac{1}{1-\xi_i(z)}.
\end{equation}

\section{Existence of the Asymptotic Range}
\label{sec:existence-range}

In this section we derive existence of the asymptotic range as formulated in Theorem \ref{thm:range-existence}. The plan of this section is as follows: in the next subsection we introduce exit times and construct a related Markov chain which tracks the random walk's trajectory to ``infinity''. While  we will derive some estimates for the increments of the range between two consecutive exit times in Subsection \ref{subsec:increments-estimates}, we finally will prove existence of $\mathfrak{r}$ in Subsection \ref{subsec:existence-range}.

\subsection{Exit Time Process}
\label{subsec:exit-time-process}
The idea is to trace the random walk's path to ``infinity''. Recall that the random walk converges to an infinite word in $V_\infty$ in the sense that the prefixes of increasing length stabilize; compare with the remarks in Subsection \ref{subsec:rw-on-fp}. For $k,n\in\mathbb{N}$, denote by $X_n^{(k)}$ the projection of $X_n$ onto its first $k$ letters. The \textit{exit times} are defined as follows: set $\e{0}:=0$ and for $k\geq 1$:
$$
\e{k}:=\inf\bigl\lbrace m> \e{k-1} \,\bigl|\, \forall n\geq m: X_n^{(k)} \textrm{ is constant}\bigr\rbrace.
$$ 
In other words, $\e{k}$ is the first instant of time from which on the first $k$ letters of $X_n$ do \textit{not} change any more. Due to \cite[Prop. 2.5]{gilch:07} we have $\e{k}<\infty$ almost surely, and by construction $C(X_{\e{k}})\subset C(X_{\e{k-1}})$.
The \textit{increments} are defined as $\mathbf{i}_k:=\e{k}-\e{k-1}$.
\par
By \cite[Proposition 3.2, Theorem 3.3]{gilch:07}, we have 
\begin{equation}\label{equ:speed}
\ell=\lim_{n\to\infty} \frac{\Vert X_n\Vert}{n}=\lim_{k\to\infty} \frac{k}{\e{k}}\in (0,1]\quad \textrm{ almost surely.}
\end{equation}
For $n\in\N_0$, define
$$
\mathbf{k}(n) := \max\{k\in\N_0 \mid \e{k}\leq n\},
$$
the \textbf{maximal exit time} at time $n$. If $X_{\e{k}}=g_1\dots g_k$, then we set 
$$
\W{k}:=g_k.
$$
Observe that 
\begin{equation}\label{equ:convergence-e_k(n)}
\lim_{k\to\infty}\frac{\e{\mathbf{k}(n)}}{n} = 1 \quad \textrm{almost surely, }
\end{equation}
since (\ref{equ:speed}) yields
$$
1\leq \frac{n}{\e{\mathbf{k}(n)}}\leq \frac{\e{\mathbf{k}(n)+1}}{\e{\mathbf{k}(n)}}=
\frac{\e{\mathbf{k}(n)+1}}{\mathbf{k}(n)}\frac{\mathbf{k}(n)}{\e{\mathbf{k}(n)}}\xrightarrow{n\to\infty} 1 \quad \textrm{almost surely}.
$$
In order to control the growth of $\mathbf{R}_n$ the idea is to count the number of visited elements of $V$ in the disjoint sets $C(X_{\e{k-1}})\setminus C(X_{\e{k}})$, which each is visited finitely often only. This partitioning needs further definitions.
We define the random function $\varphi_n:V \to\{0,1\}$ by
$$
\varphi_n(x):=\mathds{1}_{\{X_0,X_1,\dots,X_n\}}(x),
$$
that is, $\varphi_n(x)=1$ if and only if $x\in V$ is visited until time $n$. Hence, $\varphi_n$ describes the set of vertices  visited up to time $n$. The support of $\varphi_n$ is given by $\mathrm{supp}(\varphi_n)=\{x\in V\mid \varphi_n(x)=1\}$ and we have $\mathbf{R}_n=|\mathrm{supp}(\varphi_n)|$.

For $x\in V\setminus\{o\}$, we define a \textit{shift operation} denoted by $x^{-1}$ applied to subsets $A\subseteq C(x)$ as follows:
$$
x^{-1}A := \bigl\lbrace y \mid xy\in A\bigr\rbrace,
$$
that is, the shift cancels the common prefix $x$ of all elements in $A\subseteq C(x)$.
Consider now for a moment $\varphi_{\e{k}} \cdot \mathds{1}_{C(X_{\e{k-1}})}$, which has support contained in $C(X_{\e{k-1}})$. We apply a shift operation to the support of $\varphi_{\e{k}} \cdot \mathds{1}_{C(X_{\e{k-1}})}$ and define for $k\geq 2$ and $x\in V$:
$$
\psi_k(x):=\begin{cases}
1, &\textrm{if } X_{\e{k-1}}x\in C(X_{\e{k-1}}) \textrm{ and } \varphi_{\e{k}}\bigl(X_{\e{k-1}}x\bigr)=1,\\
0, &\textrm{otherwise}.
\end{cases}
$$
The support of $\psi_k$ is given by 
$$
\mathrm{supp}(\psi_k)=\{x\in V \mid \psi_k(x)=1\}=X_{\e{k-1}}^{-1} \mathrm{supp}\Bigl( \varphi_{\e{k}} \cdot \mathds{1}_{C(X_{\e{k-1}})}\Bigr).
$$
The case $k=1$ plays a distinguished role: let $V^\ast_{\t(X_{\e{1}})}$ be the set of words in $V$ starting with a letter in $V_{\t(X_{\e{1}})}$ including the empty word $o$. We define
$$
\psi_1 := \varphi_{\e{1}}\cdot \mathds{1}_{V^\ast_{\t(X_{\e{1}})}}.
$$
Let me explain the idea behind the definitions of $\psi_k$: as mentioned above we want to decompose $\mathrm{supp}(\varphi_n)$ into the disjoint parts contained in $C(X_{\e{k-1}})\setminus C(X_{\e{k}})$, $k\leq\e{\mathbf{k}(n)}$, and some remaining parts. In particular, for $n\geq \e{k}$ the set $\mathrm{supp}(\varphi_n)\cap \overline{C(X_{\e{k}})}$ remains constant, where $\overline{C(X_{\e{k}})}$ denotes the complement of the set $C(X_{\e{k}})$. More precisely, for $n\geq \e{k}$, 
$$
\mathrm{supp}(\varphi_n)\cap \overline{C(X_{\e{k}})} \subset \bigcup_{j=1}^k \bigl(C(X_{\e{j-1}})\setminus C(X_{\e{j}})\bigr),
$$
and $\varphi_n \cdot \mathds{1}_{C(X_{\e{j-1}})\setminus C(X_{\e{j}})}$ remains constant for $n\geq \e{k}$ and $j\leq k$.
The increment \mbox{$\mathbf{R}_{\e{k}}-\mathbf{R}_{\e{k-1}}$} depends only on visits of the random walk in the set $C(X_{\e{k-1}})$; thus, we just drop the common first $k-1$ already stabilized letters in the support of $\varphi_{\e{k}} \cdot \mathds{1}_{C(X_{\e{k-1}})}$ by the shift in the definition of $\psi_k$. This construction allows us to establish some homogeneity property for the already visited vertices in $C(X_{\e{k-1}})$ at time $\e{k}$, which in turn allows us to show later that $(\W{k},\psi_k)_{k\in\N}$ is a Markov chain. 

\begin{Ex}\label{ex:free-product2}
We revisit Example \ref{ex:free-product} and consider the following sample path:
$$
o \to a \to ab \to aba \to abab \to abac \to abaca  \to abac \to aba  \to abac \leadsto \textrm{staying in }C(abac).
$$
We have:
$$
X_{\mathbf{e}_1}=a, \ X_{\mathbf{e}_2}=ab, \ X_{\mathbf{e}_3}=aba, \ X_{\mathbf{e}_4}=abac.   
$$
Furthermore, we have 
\begin{eqnarray*}
&& \mathrm{supp}(\varphi_{\mathbf{e}_4}\cdot \mathds{1}_{C(X_{\mathbf{e}_3})})=\{aba,abab,abac,abaca\},\\
&&\mathrm{supp}(\psi_3)=\{o,a\}, \ \mathrm{supp}(\psi_4)=\{o,b,c,ca\}
\end{eqnarray*}
and $\mathbf{R}_{\e{4}}-\mathbf{R}_{\e{3}}=4$.
\end{Ex}

We start with the following important probability invariance property under shifts:

\begin{Lemma}\label{lem:prop-invariance}
Let be $w\in V$ and $(w,w_1,\dots,w_n)\in C(w)^{n+1}$, $n\in\N$. Then:
$$
\P_w[X_1=w_1,\dots,X_n=w_n] = \P_o[X_1=w^{-1}w_1,\dots,X_n=w^{-1}w_n].
$$
\end{Lemma}
\begin{proof}
It suffices to show that $p(w_i,w_{i+1})=p(w^{-1}w_i,w^{-1}w_{i+1})$ for $i\in \{1,\dots,n-1\}$. We write $w_i=wy_{i,1}\dots y_{i,m_i}$, where $y_{i,1},\dots,y_{i,m_i}\in V_\ast^\times$. By construction of $P$, we have
\begin{eqnarray*}
&&p(w_i,w_{i+1})\\
&=&\begin{cases}
p(y_{i,m_i},y_{i+1,m_{i+1}}), & \textrm{if } m_i=m_{i+1} \land y_{i,1}=y_{i+1,1},\dots,y_{i,m_i-1}=y_{i+1,m_{i}-1},\\
p(o,y_{i+1,m_{i+1}}), & \textrm{if } m_i+1=m_{i+1} \land y_{i,1}=y_{i+1,1},\dots,y_{i,m_i}=y_{i+1,m_{i}},\\
p(y_{i,m_{i}},o), & \textrm{if } m_i=m_{i+1}+1 \land y_{i,1}=y_{i+1,1},\dots,y_{i,m_{i}-1}=y_{i+1,m_{i+1}},\\
0, & \textrm{otherwise}
\end{cases} \\
&=&p(y_{i,1}\dots y_{i,m_i},y_{i+1,1}\dots y_{i+1,m_{i+1}}).
\end{eqnarray*}
This proves the proposed lemma.
\end{proof}

For $k\geq 1$, denote by 

$$
\mathcal{D}_k:=\bigl\lbrace (g,f) \,\bigl|\, g\in V^\times_\ast, f:V\to\{0,1\}: \P[\W{k}=g,\psi_{k}=f]>0\bigr\rbrace,
$$
the support of $(\W{k},\psi_k)$. We make the following crucial observation: 

\begin{Lemma}\label{lem:state-space}
For all $k\in\N$, $\mathcal{D}:=\mathcal{D}_1=\mathrm{supp}\bigl( (\W{k},\psi_k)\bigr)$.
\end{Lemma}
\begin{proof}
First, we prove the inclusion ``$\subseteq$''. Let be $(g,f)\in\mathcal{D}_1$ and $k\geq 2$. W.l.o.g. we assume $g\in V_1$. Take now any $g_0\in V$ with $\Vert g_0\Vert =k-1$ and $\t(g_0)=2$. Let \mbox{$(o,x_1,\dots,x_m=g_0)$} be any (shortest) path from $o$ to $g_0$ such that $m=|g_0|$. In particular, we have 
$$
\{o,x_1,\dots,x_m\}\cap C(g_0)=\{g_0\}.
$$
Furthermore, let $(o,y_1,\dots,y_n=g)$ be a path in $V\setminus V_2^\times$ such that $\mathrm{supp} (f)=\{o,y_1,\dots,y_n\}$. This choice is possible since 
$\mathrm{supp} (f)\subseteq V\setminus V_2^\times$ due to the  assumption  that $(g,f)\in\mathcal{D}_1$ holds.
 Then the path
$$
(o,x_1,\dots,x_m=g_0,g_0y_1,g_0y_2,\dots,g_0y_n=g_0g)
$$ 
is a path from $o$ to $g_0g$, which allows to generate $\W{k}=g$ and $\psi_k=f$ at time $m+n$ with positive probability, that is,
\begin{eqnarray*}
&& \mathbb{P}\bigl[(\W{k},\psi_k)=(g,f)\bigr] \\
&\geq&\mathbb{P}\left[\begin{array}{c} X_0=o,X_1=x_1,\dots,X_m=x_m=g_0,X_{m+1}=g_0y_1,\dots,X_{m+n}=g_0g,\\\forall j\geq m+n: X_j\in C(g_0g)\end{array}\right] \\
&=&\mathbb{P}\left[\begin{array}{c} X_0=o,X_1=x_1,\dots,X_m=x_m,\\
X_{m+1}=g_0y_1,\dots,X_{m+n}=g_0g\end{array}\right] \cdot \mathbb{P}_{g_0g}\bigl[\forall j\geq 1: X_j\in C(g_0g)\bigr] \\
&=& \mathbb{P}\bigl[X_0=o,X_1=x_1,\dots,X_m=g_0,X_{m+1}=g_0y_1,\dots,X_{m+n}=g_0g\bigr] \cdot (1-\xi_1) >0.
\end{eqnarray*}
Thus, we have proven $(g,f)\in \mathrm{supp}\bigl( (\W{k},\psi_k)\bigr)$.
\par
Now we show ``$\supseteq$'': let be $(g,f)\in \mathcal{D}_k$, $k\geq 2$. W.l.o.g. we assume $\t(g)=1$. Then there  exists  some $g_0\in V$ with $\Vert g_0\Vert=k-1$ and $\t(g_0)=2$ such that there exists a path of the form $(g_0,y_1,\dots,y_{m-1},y_m=g_0g)$ in $C(g_0)$ with $\mathrm{supp}(f)=\{o,g_0^{-1}y_1,\dots,g_0^{-1}y_m\}$. This implies together with Lemma \ref{lem:prop-invariance} and (\ref{equ:xi}) that 
\begin{eqnarray*}
&& \mathbb{P}\bigl[(\W{1},\psi_1)=(g,f)\bigr] \\
&\geq&\mathbb{P}\bigl[X_0=o,X_1=g_0^{-1}y_1,\dots,X_m=g_0^{-1}y_m=g,\forall j\geq m: X_j\in C(g)\bigr] \\
&=& \mathbb{P}\bigl[X_0=o,X_1=g_0^{-1}y_1,\dots,X_m=g_0^{-1}y_m\bigr] \cdot \mathbb{P}_{g}\bigr[\forall j\geq 1: X_j\in C(g)\bigr] \\
&=& \mathbb{P}_{g_0}\bigl[X_0=g_0,X_1=y_1,\dots,X_m=y_m\bigr] \cdot (1-\xi_1) >0.
\end{eqnarray*}
That is, $(g,f)\in\mathcal{D}_1$. This finishes the proof.
\end{proof}

The next proposition will be an essential ingredient in our proofs later.

\begin{Prop} \label{prop:exit-time-markov-chain}
$(\W{k},\psi_k)_{k\in\N}$ is an irreducible, homogeneous Markov chain on the state space $\mathcal{D}$.
\end{Prop}
\begin{proof}
For proving the Markov property, let  be $k\in\N$, $(g_1,f_1),\dots, (g_{k+1},f_{k+1})\in \mathcal{D}$ with 
$$
\P\bigl[(\W{1},\psi_1)=(g_1,f_1),\dots,(\W{k},\psi_k)=(g_k,f_k)\bigr]>0.
$$ 
Denote by $\Pi_{n}$, $n\in\N$, the set of paths $(o,x_1,\dots,x_{n-1},x_n=g_1\dots g_k)$ in $V$ such that
$$
\left[\begin{array}{c} X_1=x_1,\dots,X_{n-1}=x_{n-1},X_n=x_n,\\ \forall n'>n: X_{n'}\in C(x_n)\end{array}\right] \cap 
 [\e{k}=n]\cap \bigcap_{i=1}^k \bigl[(\W{k},\psi_k)=(g_k,f_k)\bigr]\neq \emptyset.
$$
Moreover, denote by $\Pi'_{m}$, $m\in\N$, the set of paths \mbox{$(o,v_1,\dots,v_{m-1},v_m=g_{k+1})$} in $V$ satisfying $v_i\notin V_j^\times$, where $j\in\calI\setminus\{\t(g_{k+1})\}$, such that
\begin{eqnarray*}
&& \bigl[X_1=v_1,\dots,X_{m}=v_m,\forall m'>m: X_{m'}\in C(v_m)\bigr]
\cap  [\e{1}=m] \neq\emptyset\textrm{ and }\\
&& \mathrm{supp}(f_{k+1}) = \{v_1,\dots,v_m\} \cup g_{k}^{-1}\bigl( \mathrm{supp}(f_{k})\cap C(g_{k})\bigr).
\end{eqnarray*}
We remark that $\mathrm{supp}(f_{k+1})$ splits up into the part $\mathrm{supp}(f_{k})\cap C(g_{k})$ visited before time $\mathbf{e}_k$ (shifted by $g_{k}^{-1}$) and into the part $\{v_1,\dots,v_m\}$ visited after time $\e{k}$.

Then we obtain by definition of $P$ and the Markov property of $(X_n)_{n\in\N_0}$:
\begin{eqnarray*}
&&\P\bigl[ (\W{1},\psi_1)=(g_1,f_1),\dots,(\W{k+1},\psi_{k+1})=(g_{k+1},f_{k+1})\bigr] \\
&=&\sum_{n\in \N}\sum_{(o,x_1,\dots,x_{n-1},x_n)\in\Pi_{n}} \P[X_1=x_1,\dots,X_n=x_n] \\
&& \quad \cdot \sum_{m\in \N}\sum_{(o,v_1,\dots,v_{m-1},v_m=g_{k+1})\in\Pi'_{m}}  \P[X_{n+1}=x_nv_1,\dots X_{n+m}=x_nv_m\mid X_n=x_n]\\[1ex]
&&\quad\quad \cdot \P[\forall j>m+n: X_j\in C(x_nv_m) \mid X_{n+m}=x_nv_m] \\[1ex]
&\stackrel{\textrm{Lemma }\ref{lem:prop-invariance}}{=}& \sum_{n\in \N}\sum_{(o,x_1,\dots,x_{n-1},x_n)\in\Pi_{n}} \P[X_1=x_1,\dots,X_n=x_n] \\
&&\quad \cdot \sum_{m\in \N}\sum_{(o,v_1,\dots,v_{m-1},v_m=g_{k+1})\in\Pi'_{m}}  \P[X_1=v_1,\dots,X_m=v_m] \cdot (1-\xi_{\t(g_{k+1})}). 
\end{eqnarray*}
Analogously,
\begin{eqnarray*}
&&\P\bigl[ (\W{1},\psi_1)=(g_1,f_1),\dots,(\W{k},\psi_{k})=(g_{k},f_{k})\bigr] \\[1ex]
&=& \sum_{n\in \N}\sum_{(o,x_1,\dots,x_{n-1},x_n)\in\Pi_{n}} \P[X_1=x_1,\dots,X_n=x_n]  \cdot (1-\xi_{\t(g_{k})}). 
\end{eqnarray*}
Thus,
\begin{eqnarray}
&& \P\bigl[(\W{k+1},\psi_{k+1})=(g_{k+1},f_{k+1})\,\bigl|\,(\W{1},\psi_1)=(g_1,f_1),\dots,(\W{k},\psi_{k})=(g_{k},f_{k}) \bigr]\nonumber \\[1ex]
&=& \frac{\P\bigl[ (\W{1},\psi_1)=(g_1,f_1),\dots,(\W{k+1},\psi_{k+1})=(g_{k+1},f_{k+1})\bigr] }{\P\bigl[ (\W{1},\psi_1)=(g_1,f_1),\dots,(\W{k},\psi_{k})=(g_{k},f_{k})\bigr]}\nonumber\\[1ex]
&=& \frac{1-\xi_{\t(g_{k+1})}}{1-\xi_{\t(g_{k})}}\cdot \sum_{m\in \N}\sum_{(o,v_1,\dots,v_{m-1},v_m=g_{k+1})\in\Pi'_{m}}  \P[X_1=v_1,\dots,X_m=v_m].\label{equ:transition-probabilities}
\end{eqnarray}
This shows that $(\W{k},\psi_k)_{k\in\N}$ is indeed a homogeneous Markov chain, since the conditional probabilities depend only on $(g_k,f_k)$ and $(g_{k+1},f_{k+1})$.
\par
For the proof of irreducibility, let be $(g_1,f_1),(g_2,f_2)\in \mathcal{D}$. From $(g_1,f_1)$ the process $(\mathbf{W}_k,\psi_k)_{k\in\N}$ can walk with positive probability in at most $|\mathrm{supp}(f_1)|+1$ steps to some $(g_0,f_0)\in\mathcal{D}$ with $g_0\notin V^\times_{\t(g_2)}$ and $f_0=\mathds{1}_o+\mathds{1}_{g_0}$. Starting at $(g_0,f_0)$ we can go to $(g_2,f_2)$ with positive probability in one step, which is easy to check and is similar to the reasoning of the proof of Lemma \ref{lem:state-space}. This yields irreducibility.
\end{proof}
\pagebreak[4]
Observe that $(\W{k},\psi_k)_{k\in\N}$ is \textit{not} aperiodic, since $\W{k}\in V_1^\times$ implies $\W{k+1}\in V_2^\times$. But we have the following result:
\begin{Lemma}
The Markov chain $(\W{k},\psi_k)_{k\in\N}$ has period $2$.
\end{Lemma}
\begin{proof}
Let be $(g,f)\in\mathcal{D}$ with $g\in V_1^\times$, $p(o,g)>0$, $f=\mathds{1}_o+\mathds{1}_g$. Take any $g'\in V_2^\times$ with $p(o,g')>0$ and set $f':=\mathds{1}_o+\mathds{1}_{g'}$. The formula (\ref{equ:transition-probabilities}) for the transition probabilities yields:
\begin{eqnarray*}
\mathbb{P}\bigl[ (\W{2},\psi_2)=(g',f')\mid (\W{1},\psi_1)=(g,f)\bigr] &\geq & \frac{1-\xi_2}{1-\xi_1}p(o,g') >0\\
\textrm{and } \quad \mathbb{P}\bigl[ (\W{3},\psi_3)=(g,f)\mid (\W{2},\psi_2)=(g',f')\bigr] &\geq & \frac{1-\xi_1}{1-\xi_2}p(o,g) >0.\\
\end{eqnarray*}
Thies implies:
\begin{eqnarray*}
&&\mathbb{P}\bigl[ (\W{3},\psi_3)=(g,f)\mid (\W{1},\psi_1)=(g,f)\bigr] \\
&\geq& \mathbb{P}\bigl[ (\W{3},\psi_3)=(g,f)\mid (\W{2},\psi_2)=(g',f')\bigr]  \\
&&\quad \cdot \mathbb{P}\bigl[ (\W{2},\psi_2)=(g',f')\mid (\W{1},\psi_1)=(g,f)\bigr] >0,
\end{eqnarray*}
that is, $(g,f)$ can be reached from $(g,f)$ within two steps, which proves the claim.
\end{proof}
We remark that the process $(\W{k},\psi_k)_{k\in\N}$ would be aperiodic if $|\mathcal{I}|\geq 3$.
\par
As an abbreviation we will denote the transition probabilities of $(\W{k},\psi_k)_{k\in\N}$ by 
$$
q\bigl((g_1,f_1),(g_2,f_2)\bigr):=\P\bigl[(\W{2},\psi_2)=(g_2,f_2)\mid (\W{1},\psi_1)=(g_1,f_1) \bigr]
$$
for $(g_1,f_1),(g_2,f_2)\in\mathcal{D}$, 
and the $n$-step transition probabilities by $q^{(n)}\bigl((g_1,f_1),(g_2,f_2)\bigr)$.

\subsection{Estimates for Increments between Exit Times}
\label{subsec:increments-estimates}

In this subsection we derive some uniform upper bounds for expectations of some random quantities associated with the exit times. Those will be needed in order to ensure that the expected increase of the range between two consecutive exit times is uniformly bounded. We start with the following lemma:

\begin{Lemma}\label{lem:sup-G'}
Let be $R_0\in (1,R)$. Then:
$$
\sup_{x\in V} G(x,x|R_0)<\infty.
$$
\end{Lemma}
\begin{proof}
Since all Green functions $G(x,x|z)$, $x\in V$, have radii of convergence of at least $R$, all values $G(x,x|R_0)$ are finite. First, we show that $\sup_{x\in V} U(x,x|1)<1$.
Let be $x\in V$ with $\t(x)=1$ and $g_2\in V_2^\times$ such that \mbox{$p_2(o_2,g_2)>0$.} 
Since $\xi_2<1$ we have:
$$
U(x,x|1)\leq 1 - \mathbb{P}_x\bigl[X_1=xg_2,\forall n>1: X_n \textrm{ has prefix } xg_2\bigr]= 1-(1-\alpha) p_2(o_2,g_2) (1-\xi_2)<1.
$$
Analogously, take any $g_1\in V_1^\times$ with $p_1(o_1,g_1)>0$. Then we obtain for all $y\in V$ with $\t(y)=2$:
$$
U(y,y|1)\leq 1-\alpha\cdot p_1(o_1,g_1)\cdot (1-\xi_1)<1.
$$
Transience yields $U(o,o|1)<1$, and therefore we have
$$
\overline{U}:=\sup_{x\in V} U(x,x|1)\leq \max\{1-\alpha p_1(o_1,g_1) (1-\xi_1),1-(1-\alpha) p_2(o_2,g_2) (1-\xi_2),U(o,o|1)\}<1.
$$

In view of (\ref{equ:Green2}) it suffices to show that $U(x,x|R_0)\in(0,1]$ is bounded away from $1$. If \mbox{$R=\infty$,} then $U(x,x|z)<1$ for all $z>0$, and monotonicity and convexity yield \mbox{$U(x,x|z)=0$.} It remains to consider the case  $R<\infty$.
First, observe that $U(x,x|R_0)<1$ for all $x\in V$ due to (\ref{equ:Green2}). Convexity of $U(x,x|z)$  together with $U(x,x|1)\leq \overline{U}$ yields 
$$
U(x,x|R_0) \leq \frac{1-\overline{U}}{R-1}\cdot R_0 +\overline{U}-\frac{1-\overline{U}}{R-1} <1,
$$
that is, $\sup_{x\in V} U(x,x|R_0)<1$, from which the claim of the lemma follows.
\end{proof}

For sake of better readability in the next proof, we introduce further notation.
For $i\in\calI$, $y\in V_j$ with $j\in\calI\setminus\{i\}$ and $n\in\N$ define
$$
k_i^{(n)}(y):=\P\bigl[\forall \ell \in\{1,\dots,n\}: X_\ell\notin V_i^\times, X_{n-1}\notin C(y),X_n=y \bigr].
$$
Furthermore, if $x\in V\setminus \{o\}$ then we set $\xi_{\neg \t(x)}=\xi_2$, if $\t(x)=1$, and $\xi_{\neg \t(x)}=\xi_1$, if $\t(x)=2$. 
\par
The following lemma ensures that the expectations of the increments between consecutive exit times are uniformly bounded.

\begin{Lemma}\label{lem:sup-i_k}
$$
\sup_{k\in\N} \mathbb{E}[\mathbf{i}_k]<\infty.
$$
\end{Lemma}
\begin{proof}
Let  be $k\in\N$. First, recall that, by definition, $X_{\e{k}-1}\notin C(X_{\e{k}})$ and that $\delta(X_{\e{k-1}})\neq \delta(X_{\e{k}})$.
We decompose $\mathbf{i}_k$ according to the values of $X_{\e{k-1}}$ and $X_{\e{k}}$:
\begin{eqnarray*}
 \mathbb{E}[\mathbf{i}_k] &=& \sum_{n\in\N} \sum_{x,y\in V} \P\bigl[ X_{\e{k-1}}=x,X_{\e{k}}=y, \e{k}-\e{k-1}=n\bigr] \cdot n \\
 &=& \sum_{n,\ell\in \N}  \sum_{\substack{x\in V:\\ \Vert x\Vert=k-1}}\sum_{\substack{y\in C(x):\\ \Vert y\Vert=k}}  \P\left[\begin{array}{c}
 X_{\ell-1}\notin C(x),X_{\ell}=x,\\
 \forall m\in\{1,\dots,n\}: X_{\ell+m}\in C(x),X_{\ell+n-1}\notin C(y),\\
 X_{\ell+n}=y, \forall t>\ell+n: X_t\in C(y)
 \end{array}\right]\cdot n \\
 &=&  \sum_{n,\ell\in \N}  \sum_{\substack{x\in V:\\ \Vert x\Vert=k-1}} \P\bigl[X_{\ell-1}\notin C(x),X_{\ell}=x\bigr] \cdot \sum_{\substack{y\in C(x):\\ \Vert y\Vert=k}} k_{\t(x)}^{(n)}(x^{-1}y)\cdot (1-\xi_{\t(y)})\cdot n \\
 &=& \sum_{\ell\in\N} \sum_{\substack{x\in V:\\ \Vert x\Vert=k-1}} \P[X_{\ell-1}\notin C(x),X_{\ell}=x] \cdot (1-\xi_{\neg \t(x)}) \\
 &&\quad \cdot \frac{\partial}{\partial z}\underbrace{ \biggl[ \sum_{n\geq 1}\sum_{y\in (V_1^\times\cup V_2^\times)\setminus V_{\t(x)}} k_{\t(x)}^{(n)}(y)\cdot  z^n \biggr]}_{=:\gamma_{\t(x)}(z)}\Biggl|_{z=1}.
\end{eqnarray*}
The power series $\gamma_{\t(x)}(z)$ has radius of convergence strictly bigger than $1$; see \cite[Proof of Prop. 3.2]{gilch:07}, where the exactly same power series are denoted by $\gamma_{i,j}(z)$. Hence,
\begin{eqnarray*}
 \mathbb{E}[\mathbf{i}_k] & \leq &\sum_{\ell\in\N}\sum_{\substack{x\in V:\\ \Vert x\Vert=k-1}} \P[X_{\ell-1}\notin C(x),X_{\ell}=x] \cdot (1-\xi_{\t(x)}) \cdot \frac{1-\xi_{\neg\t(x)}}{1-\xi_{\t(x)}} \cdot \max\{ \gamma'_{1}(1),\gamma'_{2}(1)\} \\
 & = & \underbrace{\sum_{x\in V}\P[X_{\e{k-1}}=x]}_{=1} \cdot \max\biggl\lbrace\frac{1-\xi_1}{1-\xi_2},\frac{1-\xi_2}{1-\xi_1}\biggr\rbrace \cdot \max\{ \gamma'_{1}(1),\gamma'_{2}(1)\}< \infty.
\end{eqnarray*}
This proves the claim.
\end{proof}

The following lemma will be needed to ensure that the expected increase of the range by visits of elements in $C(X_{\e{k-1}})$ up to time $\e{k}$  is uniformly bounded.
\begin{Lemma}\label{lem:sup-Ef_k}
$$
\sup_{k\in\N} \mathbb{E}\bigl[|\mathrm{supp}(\psi_k)|\bigr]<\infty.
$$
\end{Lemma}
\begin{proof}
Let  be $k\in\N$. Denote by $S_{k-1}$ the random time of the first visit to $X_{\e{k-1}}$. Then:
$$
\mathbb{E}\bigl[|\mathrm{supp}(\psi_k)|\bigr] \leq
\mathbb{E}\left[\sum_{j=S_{k-1}}^{\e{k-1}} \mathds{1}_{C(X_{\e{k-1}})}(X_j) \right] +
\mathbb{E}\left[\sum_{j=\e{k-1}+1}^{\e{k}} \mathds{1}_{C(X_{\e{k-1}})}(X_j) \right].
$$
We show that both expectations on the right hand side are uniformly bounded. To this end, we make a decomposition according to the values of $X_{\e{k-1}}$, $S_{k-1}$ and $\e{k-1}$.
For $k\in\N$, write $D_k:=\{x\in V \mid \Vert x\Vert=k\}$. Then we obtain:
\begin{eqnarray*}
&&\mathbb{E}\left[\sum_{j=S_{k-1}}^{\e{k-1}} \mathds{1}_{C(X_{\e{k-1}})}(X_j) \right] \\
&=& \sum_{m,n\in\N_0} \sum_{x\in D_{k-1}} \sum_{{x_1,\dots,x_{n-1}\in V:\atop  x_{n-1}\notin C(x)}}\P\left[ \begin{array}{c}X_m=x, \forall j<m: X_j\neq x,\\
X_{m+1}=x_1,\dots,X_{m+n-1}=x_{n-1},X_{m+n}=x, \\
\forall k>m+n: X_k\in C(x)
\end{array}
\right]\\
&&\quad \cdot \left(\mathds{1}_{C(x)}(x)+\sum_{i=1}^{n-1} \mathds{1}_{C(x)}(x_i)\right)\\
&\leq &  \sum_{m,n\in\N_0} \sum_{x\in D_{k-1}} \sum_{x_1,\dots,x_{n-1}\in V}\P\left[ \begin{array}{c}X_m=x, \forall j<m: X_j\neq x,\\
X_{m+1}=x_1,\dots,X_{m+n-1}=x_{n-1},X_{m+n}=x
\end{array}
\right]\\
&&\quad \cdot \P_x\bigl[\forall k\geq 1: X_k\in C(x)\bigr]  \cdot n
\end{eqnarray*}
\begin{eqnarray*}
&=&  \sum_{m\in\N_0} \sum_{x\in D_{k-1}} \P\bigl[  \forall j<m: X_j\neq x, X_m=x\bigr]\\
&&\quad  \cdot \sum_{n\in\N_0} 
\sum_{x_1,\dots,x_{n-1}\in V}\P_x\left[ \begin{array}{c}X_{1}=x_1,\dots, X_{n-1}=x_{n-1},\\ X_{n}=x
\end{array}
\right]\cdot (1-\xi_{\t(x)})  \cdot n\\
&= &  \sum_{m\in\N_0} \sum_{x\in D_{k-1}} \P\bigl[  \forall j<m: X_j\neq x, X_m=x\bigr] \cdot (1-\xi_{\t(x)})\cdot \sum_{n\in\N_0} n\cdot \P_x\left[X_{n}=x\right] \\
&=& \underbrace{\sum_{x\in D_{k-1}} F(o,x)\cdot (1-\xi_{\t(x)})}_{=\P[S_{k-1}=\e{k-1}]\leq 1}  \cdot G'(x,x\vert 1)
\leq \sup_{x\in V} G'(x,x\vert 1) <\infty,
\end{eqnarray*}
where the last inequality is an easy consequence of Lemma \ref{lem:sup-G'}. Due to convexity of the functions $G(x,x|z)$, $x\in V$, we must have $G'(x,x|1)\leq \sup_{x\in V} G(x,x|R_0)/(R_0-1)<\infty$.

By Lemma \ref{lem:sup-i_k},
$$
\mathbb{E}\left[\sum_{j=\e{k-1}+1}^{\e{k}} \mathds{1}_{C(X_{\e{k-1}})}(X_j) \right] 
\leq \mathbb{E}\left[\e{k}-\e{k-1}\right] = \mathbb{E}[\mathbf{i}_k] \leq  \sup_{k\in \N} \mathbb{E}[\mathbf{i}_k]<\infty.
%
$$
Thus, we have shown that
$$
\sup_{k\in\N} \mathbb{E}\bigl[|\mathrm{supp}(\psi_k)|\bigr] \leq \sup_{x\in V} G'(x,x|1)  + \sup_{k\in \N} \mathbb{E}[\mathbf{i}_k]<\infty.
$$
\end{proof}

\begin{Cor}\label{lem:conditional-Eik}
Let  $(g,f)\in \mathcal{D}$ be with $\P\bigl[(\W{1},\psi_1)=(g,f)\bigr]>0$. Then:
$$
\sup_{k\in\N} \mathbb{E}\bigl[|\mathrm{supp}(\psi_k)|\, \bigl|\, (\W{1},\psi_1)=(g,f)\bigr]<\infty.
$$
\end{Cor}
\begin{proof}
Conditioning on the event $[(\W{1},\psi_1)=(g,f)]$ gives
$$
\mathbb{E}\bigl[|\mathrm{supp}(\psi_k)|\bigr] \geq  \P\bigl[ (\W{1},\psi_1)=(g,f)\bigr] 
\cdot \mathbb{E}\bigl[|\mathrm{supp}(\psi_k)|\, \bigl|\, (\W{1},\psi_1)=(g,f)\bigr].
$$
If $\mathbb{E}\bigl[|\mathrm{supp}(\psi_k)|\, \bigl|\, (\W{1},\psi_1)=(g,f)\bigr]$ would be unbounded in $k$, then also $\mathbb{E}\bigl[|\mathrm{supp}(\psi_k)|\bigr]$ would be unbounded, a contradiction to Lemma \ref{lem:sup-Ef_k}. This proves the claim.
\end{proof}

Now we can prove the following crucial property of the process $(\W{k},\psi_k)_{k\in\N}$:

\begin{Prop}\label{prop:positive-recurrence}
$(\W{k},\psi_k)_{k\in\N}$ is positive-recurrent.
\end{Prop}
\begin{proof}
Since $(\W{k},\psi_k)_{k\in\N}$ is irreducible it is sufficient to show that  
$$
\liminf_{n\to\infty} q^{(n)}\bigl((g,f),(g,f)\bigr)>0 \quad \textrm{for some } (g,f)\in\mathcal{D}; 
$$
this follows from the fact that transience implies $\lim_{n\to\infty} q^{(n)}\bigl((g,f),(g,f)\bigr)=0$, and the rest follows from Feller \cite[Theorem on p. 389]{fellerI}.
\par
Let $(g,f)\in\mathcal{D}$ be with $g\in V_1^\times$ and $f=\mathds{1}_{\{o,g\}}$. 
Choose $M\in\N$ even such that 
$$
\mathbb{E}\bigl[|\mathrm{supp}(\psi_k)|\, \bigl|\, (\W{1},\psi_1)=(g,f)\bigr]\leq M \quad \textrm{for all } k\in\N, 
$$ 
which is possible due to Corollary \ref{lem:conditional-Eik}. Then there is some $\delta >0$ such that 
$$
p_k:=\mathbb{P}\bigl[ |\mathrm{supp}(\psi_k)|\leq M+1\,\bigl|\, (\W{1},\psi_1)=(g,f)\bigr]\geq \delta
$$ 
for all $k\in\N$: indeed, assume that there is an index sequence $(n_k)_{k\in\N}$ such that $p_{n_k}<1/k$; then 
$$
\mathbb{E}\bigl[|\mathrm{supp}(\psi_{n_k})|\, \bigl|\, (\W{1},\psi_1)=(g,f)\bigr]\geq \Bigl(1-\frac1k\Bigr)\cdot (M+1),
$$ 
a contradiction to $\mathbb{E}\bigl[|\mathrm{supp}(\psi_k)|\, \bigl|\, (\W{1},\psi_1)=(g,f)\bigr]\leq M$ for $k$ large enough. Hence, for each $m\in\N$  we can choose a set $A_m$ 
of elements $(\bar g, \bar f)\in\mathcal{D}$ with $|\mathrm{supp} (\bar f)|\leq M+1$ such that
$$
\sum_{(\bar g, \bar f)\in A_m}q^{(m)}\bigl((g,f),(\bar g, \bar f)\bigr)\geq \delta.
$$
Now let  $(\bar g,\bar f)\in\mathcal{D}$ be with $\bar g\in V_{1}^\times$ and  $|\mathrm{supp} (\bar f)|\leq M+1$, and set $(\bar g_0,\bar f_0):=(\bar g,\bar f)$ and $(\bar g_{M},\bar f_{M}):=(g, f)$. Choose $\bar g'\in V_1^\times$, $\bar g''\in V_2^\times$ with $p_1(o_1,\bar g')>0$ and $p_2(o_2,\bar g'')>0$. Set $\bar g_1:=\bar g_3:=\dots :=\bar g_{M-1}:=\bar g''$ and $\bar g_2:=\bar g_4:=\dots :=\bar g_{M-2}:=\bar g'$,
and define $\bar f_i=\mathds{1}_{B_i}$ with $B_0:=\mathrm{supp}(\bar f)$ and
$$
B_i:= \bar g_i^{-1} \bigl( B_{i-1}\cap C(\bar g_i)\bigr) \quad \textrm{ for } i\in\{1,\dots,M\}.
$$
The idea behind this  construction is to obtain a sequence of elements $(\bar g_i,\bar f_i)$ which allows to reach $(g,f)$ after $M$ steps since
$$
\mathbb{P}\bigl[ (\mathbf{W}_i,\psi_i)=(\bar g_i,\bar f_i)\,\bigl|\,  (\mathbf{W}_{i-1},\psi_{i-1})=(\bar g_{i-1},\bar f_{i-1})\bigr] 
\geq \frac{1-\xi_{\delta(\bar g_i)}}{1-\xi_{\delta(\bar g_{i-1})}} p(o,\bar g_i) >0.
$$ 
In particular, $|\mathrm{supp}(\bar f_i)|$ is reduced successively in each step, that is,
$$
|\mathrm{supp}(\bar f_0)|\geq |\mathrm{supp}(\bar f_1)|\geq |\mathrm{supp}(\bar f_2)|\geq \ldots \geq |\mathrm{supp}(\bar f_M)|=2.
$$
We set $\varepsilon_0:=\min\{\alpha\cdot  p_1(o_1,\bar g'),(1-\alpha)\cdot p_2(o_2,\bar g'')\}$ and obtain by construction of $(\bar g_i,\bar f_i)$:
\begin{eqnarray*}
 && \prod_{i=1}^{M} q\bigl( (\bar g_{i-1},\bar f_{i-1}),(\bar g_i,\bar f_i)\bigr) \\
 &=&\prod_{j=1}^{M/2} q\bigl( (\bar g_{2j-2},\bar f_{2j-2}),(\bar g_{2j-1},\bar f_{2j-1})\bigr)
 \cdot q\bigl( (\bar g_{2j-1},\bar f_{2j-1}),(\bar g_{2j},\bar f_{2j})\bigr)\\
 &=&  \prod_{j=1}^{M/2} \frac{1-\xi_2}{1-\xi_1} p(o,\bar g_{2j-1}) \cdot  \frac{1-\xi_1}{1-\xi_2} p(o,\bar g_{2j})
 \geq \varepsilon_0^{M}.
\end{eqnarray*}
Finally, we obtain for  $n\in\N$ even with $n\geq M+1$: 
\begin{eqnarray*}
q^{(n)}\bigl( (g,f),(g,f)\bigr) &=& \sum_{(\bar g, \bar f)\in\mathcal{D}} q^{(n-M)}\bigl( (g,f),(\bar g,\bar f)\bigr) \cdot q^{(M)}\bigl( (\bar g,\bar f),(g,f)\bigr)\\
&\geq & \underbrace{\sum_{(\bar g, \bar f)\in A_{n-M}} q^{(n-M)}\bigl( (g,f),(\bar g,\bar f)\bigr)}_{\geq \delta} \cdot \underbrace{q^{(M)}\bigl( (\bar g,\bar f),(g,f)\bigr)}_{\geq \varepsilon_0^{M}} 
\geq  \varepsilon_0^{M} \cdot \delta.
\end{eqnarray*}
Thus, $q^{(2n)}\bigl( (g,f),(g,f)\bigr)$, $n\geq 1+M/2$, is uniformly bounded away from zero, which proves the proposed claim. 
\end{proof}

\subsection{Existence of the Asymptotic Range}
\label{subsec:existence-range}

In this subsection we finally derive existence of the asymptotic range. Recall that the support of the random functions $\psi_k$, $k\in\N$, contain the information of the states (shifted by $X_{\e{k-1}}^{-1}$) visited by the random walk in $C(X_{\e{k-1}})$ up to time $\e{k}$. By the following definitions we decompose these sets one step further. For $k\in\mathbb{N}$, define 
$$
\tR_k:= \bigl| \mathrm{supp}(\psi_k)\setminus C(\mathbf{W}_k)\bigr|
$$ 
and the overhang 
$$
\mathbf{O}_k:=\bigl| \mathrm{supp}(\psi_k)\cap C(\mathbf{W}_k)\bigr|.
$$ 
That is, $\tR_k$ describes the final number of states  in $C(X_{\e{k-1}})\setminus C(X_{\e{k}})$ (shifted by $X_{\e{k-1}}^{-1}$) visited by the random walk, while $\mathbf{O}_k$ describes the number of states which are ``passed'' to $\psi_{k+1}$. 

\begin{Ex}
We continue Example \ref{ex:free-product2}. In view of the given sample path we have $\tR_3=\{o\}$, which counts the element $ab$,  and $\mathbf{O}_3=\{a\}$, which corresponds to the element $aba$. Furthermore, $\tR_4=\{o,c,ca\}$ represents the visited elements $aba,abac,abaca$, and $\mathbf{O}_4=\{b\}$ corresponds to the element $abab$, which will be counted later by $\tR_5$.
\end{Ex}

Our aim is to apply the ergodic theorem for positive-recurrent Markov chains to $(\mathbf{W}_k,\psi_k)_{k\in\N}$ and $(\tR_k)_{k\in\N}$, which needs the following lemma:

\begin{Lemma}\label{lem:tildeR-finite}
Let $\pi$ be the invariant probability measure of the Markov chain $(\W{k},\psi_k)_{k\in\N}$. Then:
$$
\tilde{\mathbf{r}}:=\int \tR_1 \,d\pi <\infty \quad \textrm{ and }\quad \int \mathbf{O}_1 \,d\pi <\infty.
$$
\end{Lemma}
\begin{proof}
First, observe that both integrals are well-defined since $\tR_1 \geq 0$ and $\mathbf{O}_1\geq 0$. Assume now for a moment that $\int \tR_1 \,d\pi =\infty$ holds. Then, by the ergodic theorem for positive-recurrent Markov chains and the remarks made at the beginning of Subsection \ref{subsec:exit-time-process}:
$$
\frac1n \sum_{j=1}^{\mathbf{k}(n)} \tR_j = \underbrace{\frac{\e{\mathbf{k}(n)}}{n}}_{\to 1} \underbrace{\frac{\mathbf{k}(n)}{\e{\mathbf{k}(n)}}}_{\to\ell\in(0,1]} \underbrace{\frac{1}{\mathbf{k}(n)}\sum_{j=1}^{\mathbf{k}(n)} \tR_j}_{\to\int \tR_1\,d\pi=\infty} \xrightarrow{n\to\infty} \infty \ \textrm{almost surely.}
$$
On the other hand side, the random variables $\tR_1,\ldots,\tR_{\mathbf{k}(n)}$ decompose the range $\mathbf{R}_n$ into disjoint subsets, namely into the parts contained in $C(X_{\e{j-1}})\setminus C(X_{\e{j}})$, $j\in\{1,\dots,\mathbf{k}(n)\}$. Therefore, $\sum_{j=1}^{\mathbf{k}(n)} \tR_j\leq n$, and we obtain the contradiction  
$$
\frac1n \sum_{j=1}^{\mathbf{k}(n)} \tR_j \leq 1.
$$
Thus, $\int \tR_1 \,d\pi <\infty$. The proof for finiteness of $\int \mathbf{O}_1 \,d\pi$ works completely analogously.
\end{proof}
The next corollary allows us to drop the overhang when considering the asymptotic range.
\begin{Cor}\label{cor:overhead}
$$
\lim_{k\to \infty} \frac{\mathbf{O}_k}{k}=0\quad \textrm{ almost surely}.
$$
\end{Cor}
\begin{proof}
By Lemma \ref{lem:tildeR-finite}, we get
$$
\lim_{k\to\infty} \frac{1}{k} \sum_{i=1}^k \mathbf{O}_i = \mathfrak{o}:= \int \mathbf{O}_1\,d\pi <\infty.
$$
Therefore, we obtain the proposed limit as follows:
$$
\lim_{k\to\infty} \frac{\mathbf{O}_k}{k} = \lim_{k\to\infty}\biggl[\underbrace{\frac{1}{k} \sum_{i=1}^k \mathbf{O}_i}_{\to \mathfrak{o}} - \frac{k-1}{k}\underbrace{\frac{1}{k-1} \sum_{i=1}^{k-1} \mathbf{O}_i}_{\to \mathfrak{o}}\biggr] =0 \quad \textrm{almost surely.}
$$
\end{proof}

Finally, we can prove our first main result:
\begin{proof}[Proof of Theorem \ref{thm:range-existence}:]
For $i\in\calI$, denote by $V_i^\ast$ the set of words starting with a letter in $V_i^\times$, and write $C_0:=\bigl(V_{\t(\mathbf{W}_1)}^\ast \setminus C(X_{\e{1}})\bigr)\cup\{o\}$; the complement $\overline{C_0}$ is then the set of words starting with a letter in  $V_j$, where $j\in\calI\setminus \{\t(\W{1})\}$.
Define $\tR_0:=\bigl|\{X_0,X_1,\dots,X_{\e{1}}\} \cap \overline{C_0} \bigr|$.
We decompose the set of states visited until time $\e{\mathbf{k}(n)}$ into  disjoint sets contained in \mbox{$C(X_{\e{j-1}})\setminus C(X_{\e{j}})$} with $j\in\{2,\dots,\mathbf{k}(n)\}$, $C(X_{\e{\mathbf{k}(n)}})$, $C_0$ and $\overline{C_0}$.
This gives
$$
\frac{\mathbf{R}_{\e{\k(n)}}}{\e{\k(n)}} = \underbrace{\frac{\k(n)}{\e{\k(n)}}}_{\to\ell}\cdot \biggl[\frac{1}{\k(n)}\sum_{j=1}^{\k(n)}\tR_j + \frac{1}{\k(n)}\mathbf{O}_{\k(n)}+ \frac{1}{\k(n)} \tR_0\biggr].
$$

Positive recurrence of $(\W{k},\psi_k)_{k\in\N}$ implies that, by the ergodic theorem for positive recurrent Markov chains together with Lemma \ref{lem:tildeR-finite}, 
\begin{equation}
\label{equ:R-convergence}
\lim_{n\to\infty }\frac{1}{\mathbf{k}(n)}\sum_{j=1}^{\mathbf{k}(n)}\tR_j = \tilde{\mathbf{r}} \quad \textrm{ almost surely.}
\end{equation}
Since the random walk finally converges to some infinite word in $V_\infty$ starting with a letter in $V_{\t(\mathbf{W}_1)}^\times$  and $\tR_0$ counts only finitely many elements of $V$ visited by the random walk outside of $V_{\t(\mathbf{W}_1)}^\ast\cup\{o\}$, we have $\lim_{n\to\infty}\tR_0/\mathbf{k}(n)= 0$  almost surely. This yields with Corollary \ref{cor:overhead}:
$$
\lim_{n\to\infty}\frac{\mathbf{R}_{\e{\k(n)}}}{\e{\k(n)}}  =\tilde{\mathbf{r}}\cdot \ell=:\mathfrak{r}>0\quad \textrm{ almost surely.}
$$

Furthermore, we obtain with (\ref{equ:convergence-e_k(n)}):
\begin{eqnarray*}
\frac{\mathbf{R}_n}{n} &=& \frac{\mathbf{R}_n-\mathbf{R}_{\e{\k(n)}}}{n}+\underbrace{\frac{\mathbf{R}_{\e{\k(n)}}}{\e{\k(n)}}}_{\to \mathfrak{r}}\underbrace{\frac{\e{\k(n)}}{n}}_{\to 1}.
\end{eqnarray*}
It remains to show that the first quotient on the right hand side tends to $0$, which is obtained as follows from (\ref{equ:speed}) and (\ref{equ:convergence-e_k(n)}):
\begin{eqnarray*}
0&\leq & \frac{\mathbf{R}_n-\mathbf{R}_{\e{\k(n)}}}{n} \leq \frac{\mathbf{R}_{\e{\k(n)+1}}-\mathbf{R}_{\e{\k(n)}}}{n}
=  \frac{\mathbf{R}_{\e{\k(n)+1}}-\mathbf{R}_{\e{\k(n)}}}{\e{\k(n)+1}}\frac{\e{\k(n)+1}}{n} \\
&=& \biggl[ \underbrace{\frac{\mathbf{R}_{\e{\k(n)+1}}}{\e{\k(n)+1}}}_{\to \mathfrak{r}}-\underbrace{\frac{\mathbf{R}_{\e{\k(n)}}}{\e{\k(n)}}}_{\to \mathfrak{r}}\underbrace{\frac{\e{\k(n)}}{\k(n)}}_{\to1/\ell}\underbrace{\frac{\k(n)}{\e{\k(n)+1}}}_{\to \ell}\biggr] \underbrace{\frac{\e{\k(n)+1}}{n}}_{\to 1}
\xrightarrow{n\to\infty} 0 \quad \textrm{ almost surely}.
\end{eqnarray*}
This completes the proof.
\end{proof}

The proof of Theorem \ref{thm:range-existence} gives the following formula for the asymptotic range:
\begin{Cor}
$$
\mathfrak{r}=\lim_{n\to\infty} \frac{\mathbf{R}_n}{n} = \tilde{\mathbf{r}}\cdot\ell \quad \textrm{ almost surely}.
$$
\end{Cor}
\begin{flushright}$\Box$\end{flushright}

\begin{Remark}\normalfont
The formula (\ref{equ:group-case-formula}) for the asymptotic range in the group setting does, in general, \textit{not} necessarily hold for general free products of graphs as we will demonstrate in the following example. We consider an adapted version of Example \ref{ex:free-product}: let be \mbox{$V_1=\{o_1,a\}$,} $V_2=\{o_2,b,c\}$, set $p_1(o_1,a)=p_1(a,a)=1$ and $p_2(o_2,b)=p_2(b,c)=p_2(c,b)=1$ and $\alpha=\frac12$. Then $U(o,o)=0$ by construction, that is, the formula (\ref{equ:group-case-formula}) for the asymptotic range in the group case gives the value $1-U(o,o|1)=1$. Now we show that the asymptotic range is strictly smaller than $1$. 
\par
For $x\in V_1\ast V_2$, $x\neq o$, denote by $[x]$ the \textit{last} letter of $x$.
Since $o\in V_1\ast V_2$ can \textit{not} be visited again, we have for $n\in\N$:
$$
\P\bigl[[X_n]=a\bigr]+\P\bigl[[X_n]=b\bigr]+\P\bigl[[X_n]=c\bigr]=1.
$$
For $n\in\N$, define the events
\begin{eqnarray*}
A_n &= & \big[[X_{4n}]=a,[X_{4n+1}]=b,X_{4n+1}=X_{4n+3}\bigr] \cup \big[[X_{4n}]=b,X_{4n}=X_{4n+2}\bigr]\\
&&\cup \big[[X_{4n}]=c,X_{4n}=X_{4n+2}\bigr].
\end{eqnarray*}
Then:
\begin{eqnarray*}
\P(A_n) &=& \P\bigl[[X_{4n}]=a\bigr] \cdot \P\bigl[ [X_{4n+1}]=b,X_{4n+1}=X_{4n+3} \mid [X_{4n}]=a\bigr] \\
&& + \P\bigl[[X_{4n}]=b\bigr] \cdot \P\bigl[ X_{4n}=X_{4n+2} \mid [X_{4n}]=b\bigr] \\
&& + \P\bigl[[X_{4n}]=c\bigr] \cdot \P\bigl[ X_{4n}=X_{4n+2} \mid [X_{4n}]=c\bigr]\\
&\geq&  \P\bigl[[X_{4n}]=a\bigr] \cdot \Bigl(\frac12\Bigr)^3 + \P\bigl[[X_{4n}]=b\bigr] \cdot \Bigl(\frac12\Bigr)^2 + \P\bigl[[X_{4n}]=c\bigr] \cdot \Bigl(\frac12\Bigr)^2 \geq  \Bigl(\frac12\Bigr)^3.
\end{eqnarray*}
If the event $A_n$ occurs then the random walk visits at least one element of $V$ in the time interval $[4n,4n+3]$ twice, that is, $|\{X_{4n},X_{4n+1},X_{4n+2},X_{4n+3}\}|<4$. Hence, each occurrence of $A_n$ reduces the maximal value of $\mathbf{R}_n$ by at least $1$. Since $\mathbf{R}_{4n}\leq 4n$ and the $A_n$'s cover disjoint time slots, we obtain:
$$
\mathbb{E}[\mathbf{R}_{4n}]  \leq  4n- \sum_{k=1}^n \mathbb{E}[\mathds{1}_{A_n}] = 4n -  \sum_{k=1}^n \P(A_n) 
\leq   4n- \sum_{k=1}^n \frac{1}{8}=4n+\frac{n}{8}.
$$
That is, an application of the Dominated convergence Theorem gives
$$
\mathfrak{r} = \lim_{n\to\infty} \frac{\mathbf{R}_n}{n} 
=\lim_{n\to\infty} \frac{\mathbb{E}[\mathbf{R}_{4n}] }{4n} \leq \lim_{n\to\infty} 1- \frac{\frac{n}{8}}{4n}= 1- \frac{1}{32} <1.
$$
Hence, the asymptotic range is strictly smaller than $1$, and the formula from the group setting does \textit{not} hold.
\end{Remark}

\section{Central Limit Theorem}
\label{sec:clt}

In this section we derive the Central Limit Theorem \ref{thm:clt}. The idea is to decompose the set of visited vertices up to time $n$ into disjoint i.i.d. subsets with the help of regeneration times. For this purpose, fix for the rest of this section any $g_0\in V_1^\times$ with $p_1(o_1,g_0)>0$ and define the following random times (recall that $S_x$ is the random time of the first visit in $x\in V$):
$$
\tau_0:=\inf\bigl\lbrace m\in \N \,\bigl|\, \mathbf{W}_{\e{m}}=g_0,\e{m}=S_{X_{\e{m}}}\bigr\rbrace
$$
and for $k\geq 1$
$$
\tau_k:=\inf\bigl\lbrace m > \tau_{k-1} \,\bigl|\, \mathbf{W}_{\e{m}}=g_0,\e{m}=S_{X_{\e{m}}}\bigr\rbrace.
$$
Due to positive-recurrence of $(\W{k},\psi_k)_{k\in\N}$ (Proposition \ref{prop:positive-recurrence}), we have $\mathbf{W}_{m}=g_0$ for infinitely many indices $m\in\N$ with probability $1$. Each time when the random visits for the first time some word $w\in V$ ending with $g_0$, it has probability of $\xi_1>0$ to stay in the cone $C(w)$ thereafter. This observation gives rise to a standard geometric argument which yields $\tau_k<\infty$ almost surely. In particular, the definition of $\tau_k$ implies 
$$
g_0^{-1}\bigl( \mathrm{supp}(\psi_{\mathbf{e}_{\tau_k}}) \cap C(g_0)\bigr)=\{o\}.
$$
For $i\in\N_0$, set
\begin{equation}\label{equ:def-Ti}
T_i:=\e{\tau_i},
\end{equation}
and define
\begin{equation}\label{def:Li}
\widetilde{\L}_i:= \sum_{j=\tau_{i-1}+1}^{\tau_i} \widetilde{\mathbf{R}}_j, \quad
\L_i:= \widetilde{\L}_i-(T_i-T_{i-1})\cdot \mathfrak{r} =(\mathbf{R}_{T_i}-\mathbf{R}_{T_{i-1}})-(T_i-T_{i-1})\cdot \mathfrak{r};
\end{equation}
the last equation uses the fact that $\mathbf{O}_{T_{i-1}}=\mathbf{O}_{T_i}=\mathds{1}_{\{o\}}$ by construction of $T_{i-1}$ and $T_i$.
In the following we will show that $T_1-T_0$ and $T_0$  have exponential moments and that  $(\L_i)_{i\in\N}$ and $(T_i-T_{i-1})_{i\in\N}$ form i.i.d. sequences with finite second moment.
\par
To this end we have to introduce further notation. Set
$$
V_{g_0} :=\bigl\lbrace w_1\dots w_n\in V\,\bigl|\, n\in\N,w_1,\dots,w_{n-1}\neq g_0, w_n=g_0\bigr\rbrace,
$$
the set of words in $V$ which end with letter $g_0$ and have no further occurrence of this letter, and
$$
V_{g_0}^{(2)}:=\bigl\lbrace w\in V_{g_0}\,\bigl|\, w \textrm{ starts with a letter in } V_2^\times\bigr\rbrace.
$$
Define for $z\in\mathbb{C}$
$$
\mathcal{L}(z):=\sum_{w\in V_{g_0}} L(o,w|z)
$$
and 
\begin{eqnarray*}
\mathcal{L}_1^{\times}(z) &:=& \sum_{g\in V_1^\times\setminus\{g_0\}} L(o,g|z) \stackrel{(\ref{equ:L-Li})}{=}
\sum_{g\in V_1^\times\setminus\{g_0\}} L_1\bigl(o_1,g\,\bigl| \, \xi_1(z)\bigr), \\
\mathcal{L}_2^{\times}(z) &:=& \sum_{g\in V_2^\times} L(o,g|z) \stackrel{(\ref{equ:L-Li})}{=}
\sum_{g\in V_2^\times } L_2\bigl(o_2,g\,\bigl| \, \xi_2(z)\bigr).
\end{eqnarray*}
Our next goal is to prove that $\mathcal{L}(z)$ has radius of convergence strictly bigger than $1$, from which we can deduce existence of exponential moments of $T_1-T_0$.
With the help of (\ref{equ:L-factorisation}) we can rewrite $\mathcal{L}(z)$ as follows:
\begin{eqnarray*}
\mathcal{L}(z) &=& \sum_{\substack{n\in\N,\\ w_1\dots w_n\in V_{g_0}}} \prod_{j=1}^n L_{\t(w_j)}\bigl( o_{\t(w_j)},w_j\,\bigl|\, \xi_{\t(w_j)}(z)\bigr) \\
&=& L\bigl(o,g_0|z) \cdot \Big( 1+ \sum_{n\geq 1} \bigl(\mathcal{L}_1^{\times}(z)\cdot \mathcal{L}_2^{\times}(z)\bigr)^n + \mathcal{L}_2^{\times}(z) \cdot \sum_{n\geq 0} \bigl(\mathcal{L}_1^{\times}(z)\cdot \mathcal{L}_2^{\times}(z)\bigr)^n\Bigr).
\end{eqnarray*}
In the last equation we made a case distinction whether a word $w_1\dots w_n\in V_{g_0}$ equals $g_0$, starts with any letter $w_1\in V_1^\times$ or with any letter $w_2\in V_2^\times$; also recall that $w_i\in V_1^\times$ ($w_i\in V_2^\times$ respectively) implies $w_{i+1}\in V_2^\times$ ($w_{i+1}\in V_1^\times$ respectively). For $z\in\mathbb{C}$ within the disc of convergence, we simplify the above formula for $\mathcal{L}(z)$:
\begin{eqnarray}
\mathcal{L}(z) &=& L_1\bigl(o_1,g_0\,\bigl| \, \xi_1(z)\bigr) \cdot \big( 1+ \mathcal{L}_2^{\times}(z)\bigr) \cdot  \sum_{n\geq 0} \bigl(\mathcal{L}_1^{\times}(z)\cdot \mathcal{L}_2^{\times}(z)\bigr)^n \nonumber \\
&=&  L_1\bigl(o_1,g_0\,\bigl| \, \xi_1(z)\bigr) \cdot \frac{1+ \mathcal{L}_2^{\times}(z)}{1-\mathcal{L}_1^{\times}(z)\cdot \mathcal{L}_2^{\times}(z)}.\label{equ:calL-equation}
\end{eqnarray}
Moreover, we can rewrite $\mathcal{L}_1^{\times}(z)$ with the help of (\ref{equ:L-G}) and (\ref{equ:Gi-sum}) as follows:
\begin{eqnarray*}
\mathcal{L}_1^{\times}(z)  
&=& \sum_{g\in V_1^\times\setminus\{g_0\}}\frac{G_1\bigl(o_1,g\,\bigl|\, \xi_1(z)\bigr)}{G_1\bigl(o_1,o_1\,\bigl|\,\xi_1(z)\bigr)} \\
&=& \frac{1}{\bigl(1-\xi_1(z)\bigr)\cdot G_1\bigl(o_1,o_1\,\bigl|\,\xi_1(z)\bigr)}- \frac{G_1\bigl(o_1,g_0\,\bigl|\,\xi_1(z)\bigr)}{G_1\bigl(o_1,o_1\,\bigl|\,\xi_1(z)\bigr)}-1.
\end{eqnarray*}
Analogously,
\begin{eqnarray*}
\mathcal{L}_2^{\times}(z)  &=&  \sum_{g\in V_2^\times}\frac{G_2\bigl(o_2,g\,\bigl|\,\xi_2(z)\bigr)}{G_2\bigl(o_2,o_2\,\bigl|\,\xi_2(z)\bigr)} 
= \frac{1}{\bigl(1-\xi_2(z)\bigr)\cdot G_2\bigl(o_2,o_2\,\bigl|\,\xi_2(z)\bigr)}-1.
\end{eqnarray*}
From the  formulas above follows that $\mathcal{L}_1^{\times}(z)$ and $\mathcal{L}_2^{\times}(z)$ both have radii of convergence strictly bigger than $1$, since $\xi_1(z), \xi_2(z)$ have radii of convergence strictly bigger than $1$ and satisfy $\xi_1(1),\xi_2(1)<1$; see \cite[Lemma 2.3]{gilch:07}.
\begin{Lemma}\label{lem:L-diff}
$\mathcal{L}(z)$ has radius of convergence $R(\mathcal{L})>1$.
\end{Lemma}
\begin{proof}
First, we recall that if $X_{T_j}=x$ then $X_{T_j-1}\notin C(x)$ by definition of $T_i$.
Due to positive-recurrence of $(\mathbf{W}_k)_{k\in\N}$ (immediate consequence of Proposition \ref{prop:positive-recurrence}) we have
\begin{eqnarray*}
1 &=& \P\bigl[\exists k\in\N: \mathbf{W}_k=g_0\bigr] \\
&=& \sum_{w\in V_{g_0}} \sum_{g\in V_2^\times} \P\bigl[X_{e_{\Vert w\Vert}}=w,X_{e_{\Vert w\Vert+1}}=wg\bigr]\\
&=& \sum_{w\in V_{g_0}} G(o,w|1)\\
&&\quad \cdot \sum_{g\in V_2^\times,m\in\N} \P_{w}\bigl[X_m=wg,X_{m-1}\notin C(wg),\forall k<m: X_k\neq w\bigr]\cdot \bigl(1-\xi_2\bigr)\\
&\stackrel{(\ref{equ:Green3})}{=}& G(o,o|1)\cdot \mathcal{L}(1) \\
&&\quad \cdot \sum_{g\in V_2^\times,m\in\N} \P_{o}\bigl[X_m=g,X_{m-1}\notin C(g),\forall k<m: X_k\neq o\bigr]\cdot \bigl(1-\xi_2\bigr).
\end{eqnarray*}
In the last equation we have used once again that there is measure-preserving bijection of paths in $C(w)$ to paths in $V$ not visiting $V_1^\times$ by the shift $C(w) \ni w' \mapsto w^{-1} w'$.
\par
The  equations above imply that $\mathcal{L}(1)<\infty$, and in particular $\mathcal{L}_1^{\times}(1)\cdot \mathcal{L}_2^{\times}(1)<1$. Since $\mathcal{L}_1^{\times}(1), \mathcal{L}_2^{\times}(1)$ are continuous and both have radii of convergence strictly bigger than $1$, Equation (\ref{equ:calL-equation}) provides that $\mathcal{L}(z)$ has also radius of convergence $R(\mathcal{L})$ strictly bigger \mbox{than $1$.}
\end{proof}
The last lemma implies that there exists
$R_0\in \bigl(1,R(\mathcal{L})\bigr)$ such that $\mathcal{L}(z)$ is arbitrarily often differentiable at $z=R_0$. This fact allows us to prove  the following proposition. Recall the definition of $T_0$ which implies $T_0=S_{X_{T_0}}$ and $\mathbf{O}_{T_0}=\{o\}$.

\begin{Prop}\label{lem:increment-bound}
The power series
$$
\sum_{n\geq 1} \P[T_1-T_0=n]\cdot z^n
$$
has radius of convergence strictly bigger than $1$. In particular, $T_1-T_0$ has exponential moments.
\end{Prop}
\begin{proof}
We rewrite the power series under consideration for $z>0$ by decomposing according to the values of $T_0$, $T_1-T_0$, $X_{T_0}$ and $X_{T_1}$:
\begin{eqnarray*}
&&\sum_{n\geq 1} \P[T_1-T_0=n]\cdot z^n\\
&=& \sum_{w\in V_{g_0}} \sum_{y\in V_{g_0}^{(2)}}  \sum_{m,n\in\N} \mathbb{P}\bigl[T_0=m,X_m=w,T_1=m+n,X_{m+n}=wy \bigr]\\
&\leq& \sum_{w\in V_{g_0}} \sum_{m\in\N} \P\bigl[ X_m=w,\forall m'<m: X_{m'}\neq w\bigr] \\
&&\quad \cdot \sum_{j\in \N_0} \P_w\bigl[X_j=w,\forall j'<j: X_{j'}\in C(w)\bigr]\cdot z^j  \\
&& \quad \cdot \sum_{y\in V_{g_0}^{(2)}} \sum_{l\in\N} \P_w\bigl[X_l=wy,\forall l'<l: X_{l'}\notin C(wy)\cup\{w\}\bigr] \cdot z^l  \cdot (1-\xi_1) \\
&\leq & \sum_{w\in V_{g_0}} \sum_{m\in\N} \P\bigl[ X_m=w,\forall m'<m: X_{m'}\neq w\bigr] \\
&&\quad \cdot \sum_{j\in \N_0} \P_w\bigl[X_j=w\bigr]\cdot z^j   \cdot \sum_{y\in V_{g_0}^{(2)}} \sum_{l\in\N} \P_w\bigl[X_l=wy,\forall l'<l: X_{l'}\neq w\bigr] \cdot z^l  \cdot (1-\xi_1) \\
&= & \sum_{w\in V_{g_0}} \sum_{m\in\N} \P\bigl[ X_m=w,\forall m'<m: X_{m'}\neq w\bigr] \cdot (1-\xi_1)\\
&&\quad \cdot \sum_{j\in \N_0} \P_w\bigl[X_j=w\bigr]\cdot z^j  \cdot \sum_{y\in V_{g_0}^{(2)}} \sum_{l\in\N} \P_o\bigl[X_l=y,\forall l'<l: X_{l'}\neq o\bigr] \cdot z^l,
\end{eqnarray*}
where we applied in the last step the same shift transformation as in the proof of Lemma \ref{lem:L-diff}. Moreover, we remark that 
\begin{eqnarray*}
&&\sum_{w\in V_{g_0}} \sum_{m\in\N} \P\bigl[ X_m=w, \forall m'<m: X_{m'}\neq w\bigr] \cdot (1-\xi_1) \\
&=& \P\bigl[\exists n\in\N: \W{n}=g_0,\e{n}=S_{X_{\e{n}}}\bigr].
\end{eqnarray*}
This yields for any $R_0\in\bigl(1,\min\{R,R(\mathcal{L})\}\bigr)$ together with Lemma \ref{lem:sup-G'}:
\begin{eqnarray*}
&&\sum_{n\geq 1} \P[T_1-T_0=n]\cdot R_0^n\\
&\leq & \sup_{w\in V}\sum_{j\in \N_0} \P_w\bigl[X_j=w\bigr]\cdot R_0^j  \cdot \underbrace{\sum_{y\in V_{g_0}^{(2)}} \sum_{l\in\N} \P\bigl[X_l=y,\forall l'<l: X_{l'}\neq o\bigr] \cdot R_0^l}_{\leq \mathcal{L}(R_0)}\\
&\leq & \sup_{w\in V} G(w,w|R_0) \cdot \mathcal{L}(R_0) <\infty.
\end{eqnarray*}
Since $\sum_{n\geq 1} \P[T_1-T_0=n]\cdot z^n$ has non-negative coefficients, Pringsheim's Theorem yields the proposed claim.

\end{proof}

Analogously, we have:
\begin{Lemma}\label{lem:expT0-finite}
$T_0$ has exponential moments.
\end{Lemma}
\begin{proof}
We show that $\sum_{n\geq 1} \P[T_0=n]\cdot z^n$ has radius of convergence strictly bigger than $1$, from which the claim follows. Rewriting this power series gives:
\begin{eqnarray*}
 \sum_{n\geq 1} \P[T_0=n]\cdot z^n 
&=& \sum_{n\geq 1} \sum_{w\in V_{g_0}} \P\bigl[X_n=w, \forall l<n: X_{l}\notin C(w),\forall m>n: X_{m}\in C(w)\bigr]\cdot z^n\\
&\leq & \sum_{n\geq 1} \sum_{w\in V_{g_0}} \P\bigl[X_n=w\bigr]\cdot \P_w\bigl[\forall m\geq 1: X_{m}\in C(w)\bigr]\cdot z^n\\
&=& \sum_{w\in V_{g_0}} G(o,o|z) \cdot L(o,w|z) \cdot (1-\xi_1) = G(o,o|z) \cdot\mathcal{L}(z)\cdot (1-\xi_1).
\end{eqnarray*}
The claim follows now from Lemma \ref{lem:L-diff}.
\end{proof}

The last proposition implies the following  important corollary:
\begin{Cor}
$\sigma_{\L}^2:=\mathrm{Var}(\L_1) <\infty$.
\end{Cor}
\begin{proof}
The claim follows from
$$
0\leq \widetilde{\L}_1 = \mathbf{R}_{T_1}-\mathbf{R}_{T_0} \leq T_1-S_{X_{T_0}}=T_1-T_0
$$
and Proposition \ref{lem:increment-bound}.
\end{proof}

The following proposition demonstrates that the random times $T_i$, $i\in\N$, are regeneration times:
\begin{Prop}\label{prop:iid-sequence}
$(T_i-T_{i-1})_{i\in\N}$ and
$(\L_i)_{i\in\N}$ are i.i.d. sequences of random variables.
\end{Prop}
\begin{proof}
In the first step we show that the $\L_i$'s are identically distributed. 
Let be $i,m\in\N$, $j\in\N_0$ and $z\in\mathbb{R}$. For $x_0\in V$ with $\P[X_{\e{\tau_j}}=x_0]>0$, denote by $\mathcal{P}^{(1)}_{j,x_0,m}$ the set of  all paths \mbox{$(o,w_1,\dots,w_m=x_0)\in V^{m+1}$} of length $m$ such that 
$$
\bigl[X_1=w_1,\dots,X_{m-1}=w_{m-1},X_m=x_0\bigr]\cap \bigl[X_m=x_0,\e{\tau_j}=m\bigr]\neq \emptyset,
$$
that is, each path in $\mathcal{P}^{(1)}_{j,x_0,m}$ allows to generate $\e{\tau_j}$ with $X_{\e{\tau_j}}=x_0$ at time $m$.
 
Furthermore, for $x_0\in V$ with $\P[X_{\e{\tau_{i-1}}}=x_0]>0$, denote by $\mathcal{P}^{(2)}_{i,x_0,n,z}$ the set of paths $(x_0,y_1,\dots,y_n)\in V^{n+1}$ of length $n\in\N$ such that 
$$
\bigcup_{t\in\N} \bigl[X_{t}=x_0,X_{t+1}=y_1,\dots,X_{t+n}=y_n\bigr]\cap \bigl[X_{t}=x_0,\e{\tau_{i-1}}=t,\e{\tau_{i}}=t+n,\L_i=z\bigr]\neq \emptyset,
$$
that is, each path in $\mathcal{P}^{(2)}_{i,x_0,n,z}$ allows to generate $X_{\e{\tau_{i-1}}}=x_0$, $\e{\tau_{i}}-\e{\tau_{i-1}}=n$ and $\L_i=z$. In particular, we have $y_i\in C(x_0)$.
By decomposing all paths until time $\e{\tau_{i}}$ into the part until time $\e{\tau_{i-1}}$ and into the part between the random times $\e{\tau_{i-1}}$ and $\e{\tau_{i}}$  we obtain with Lemma \ref{lem:prop-invariance}:
\begin{eqnarray*}
\P[\L_i=z] &=& \sum_{\substack{x_0\in V:\\ \P[X_{\e{\tau_{i-1}}}=x_0]>0}} \P\bigl[ X_{\e{\tau_{i-1}}}=x_0,\L_i=z\bigr]\\
&=& \sum_{\substack{x_0\in V:\\ \P[X_{\e{\tau_{i-1}}}=x_0]>0}}  \sum_{\substack{ n\in\N,\\ (x_0,y_1,\dots,y_n)\in \mathcal{P}^{(2)}_{i,x_0,n,z}}}  \P\left[ \begin{array}{c} X_{\e{\tau_{i-1}}}=x_0,\\ X_{\e{\tau_{i-1}}+1}=y_1,\dots, X_{\e{\tau_{i-1}}+n}=y_n,\\ \forall l\geq 1: X_l \in C(y_n)\end{array}
\right]\\
&=&  \sum_{\substack{x_0\in V:\\ \P[X_{\e{\tau_{i-1}}}=x_0]>0}} \sum_{m\geq 1} \sum_{(o,w_1,\dots,w_m)\in \mathcal{P}^{(1)}_{i-1,x_0,m}} \P\bigl[X_1=w_1,\dots,X_m=w_m\bigr] \\
&&\quad \cdot \sum_{n\geq 1} \sum_{(x_0,y_1,\dots,y_n)\in \mathcal{P}^{(2)}_{i,x_0,n,z}} \P_{x_0}\bigl[X_1=y_1,\dots,X_n=y_n\bigr] \\
&&\quad \cdot \P_{y_n}\bigl[\forall l\geq 1: X_l \in C(y_n)\bigr]\\
&\stackrel{L.\ref{lem:prop-invariance}}{=}&\sum_{\substack{x_0\in V:\\ \P[X_{\e{\tau_{i-1}}}=x_0]>0}} \sum_{m\geq 1} \sum_{(o,w_1,\dots,w_m)\in \mathcal{P}^{(1)}_{i-1,x_0,m}} \P\bigl[X_1=w_1,\dots,X_m=w_m\bigr] \\
&&\quad \cdot \sum_{n\geq 1} \sum_{(x_0,y_1,\dots,y_n)\in \mathcal{P}^{(2)}_{i,x_0,n,z}} \P_{g_0}\bigl[X_1=g_0x_0^{-1}y_1,\dots,X_n=g_0x_0^{-1}y_n\bigr] \cdot \bigl(1-\xi_1\bigr).
\end{eqnarray*}
Observe  that paths $(x_0,y_1,\dots,y_n)\in \mathcal{P}^{(2)}_{i,x_0,n,z}$ lie completely in the set $C(x_0)$. Therefore, there is a natural 1-to-1 correspondence between paths in $\mathcal{P}^{(2)}_{i,x_0,n,z}$ and $\mathcal{P}^{(2)}_{1,g_0,n,z}$ established by the vertex-wise shift $C(x_0)\ni g \mapsto g_0x_0^{-1}g\in C(g_0)$, which allows an application of Lemma \ref{lem:prop-invariance}. Moreover,
\begin{eqnarray*}
&&\sum_{\substack{x_0\in V:\\ \P[X_{\e{\tau_{i-1}}}=x_0]>0}} \sum_{m\geq 1} \sum_{(o,w_1,\dots,w_m)\in \mathcal{P}^{(1)}_{i-1,x_0,m}} \P\bigl[X_1=w_1,\dots,X_m=w_m\bigr] \cdot (1-\xi_1)\\
&=& \sum_{\substack{x_0\in V:\\ \P[X_{\e{\tau_{i-1}}}=x_0]>0}} \P\bigl[X_{\e{\tau_{i-1}}}=x_0\bigr]
=\P[\mathbf{e}_{\tau_{i-1}}<\infty]=1.
\end{eqnarray*}
Therefore,
\begin{eqnarray}\label{equ:Li-formula}
\P[\L_i=z] = \sum_{n\geq 1} \sum_{(g_0,y_1,\dots,y_n)\in \mathcal{P}^{(2)}_{1,g_0,n,z}} \P_{g_0}\bigl[X_1=y_1,\dots,X_n=y_n\bigr].
\end{eqnarray}
This proves that the $\L_i$'s all have the same distribution. Analogously, we can show that, for $k\in\N$,
\begin{eqnarray*}
&&\P[T_i-T_{i-1}=k] \\
&=& \sum_{\substack{x_0\in V:\\ \P[X_{\e{\tau_{i-1}}}=x_0]>0}}  \sum_{(x_0,y_1,\dots,y_k)\in \bigcup_{z\in\mathrm{supp}(\L_i)} \mathcal{P}^{(2)}_{i,x_0,k,z}}  \P\left[ \begin{array}{c} X_{\e{\tau_{i-1}}}=x_0,\\ X_{\e{\tau_{i-1}}+1}=y_1,\dots, X_{\e{\tau_{i-1}}+k}=y_k,\\ \forall l\geq 1: X_l \in C(y_k)\end{array}
\right]\\
&=& \sum_{(g_0,y_1,\dots,y_k)\in\bigcup_{z\in\mathrm{supp}(\L_i)} \mathcal{P}^{(2)}_{1,g_0,k,z}} \P_{g_0}\bigl[X_1=y_1,\dots,X_k=y_k\bigr],
\end{eqnarray*}
which proves that $(T_i-T_{i-1})_{i\in\N}$ is identically distributed.
\par
We restrict ourselves to the proof of independence of $\L_1$ and $\L_2$ for sake of better readability; the general proof follows completely analogously and is therefore omitted. Let be $l_1,l_2\in\mathbb{R}$. Then:
\begin{eqnarray}
&&\P[\L_1=l_1,\L_2=l_2] \nonumber\\
&=&  \sum_{\substack{x_0\in V:\\ \P[X_{\e{\tau_{0}}}=x_0]>0}}\sum_{\substack{y\in V:\\ \P[X_{\e{\tau_{1}}}=y]>0}}\sum_{\substack{z\in V:\\ \P[X_{\e{\tau_{2}}}=z]>0}}\sum_{m,n_1,n_2\geq 1}
\P\left[\begin{array}{c} e_{\tau_0}=m, X_{e_{\tau_0}}=x_0,\\  e_{\tau_1}=m+n_1, X_{e_{\tau_1}}=y, \\  e_{\tau_2}=m+n_1+n_2, X_{e_{\tau_2}}=z,\\ L_1=l_1, L_2=l_2\end{array}\right]\nonumber\\
&=& \sum_{\substack{x_0\in V:\\ \P[X_{\e{\tau_{0}}}=x_0]>0}}\sum_{m,n_1,n_2\geq 1} \sum_{\substack{(o,w_1,\dots,w_m)\in \mathcal{P}^{(1)}_{0,x_0,m}, \\ (x_0,y_1,\dots,y_{n_1})\in \mathcal{P}^{(2)}_{1,x_0,n_1,l_1},\\ (y_{n_1},z_1,\dots,z_{n_2})\in \mathcal{P}^{(2)}_{2,y_{n_1},n_2,l_2}}} \P\left[\begin{array}{c}
X_1=w_1,\dots,X_m=w_m,\\
X_{m+1}=y_1,\dots,X_{m+n_1}=y_{n_1},\\
X_{m+n_1+1}=z_1,\dots,\\ X_{m+n_1+n_2}=z_{n_2},\\
\forall l\geq 1: X_l\in C(z_{n_2})
\end{array}\right]\nonumber\\
&=& \sum_{\substack{x_0\in V:\\ \P[X_{\e{\tau_{0}}}=x_0]>0}} \sum_{m\geq 1} \sum_{(o,w_1,\dots,w_m)\in \mathcal{P}^{(1)}_{0,x_0,m}} \P\bigl[X_1=w_1,\dots,X_m=w_m\bigr]\nonumber \\
&&\quad \cdot \sum_{n_1\geq 1} \sum_{(x_0,y_1,\dots,y_{n_1})\in \mathcal{P}^{(2)}_{1,x_0,n_1,l_1}} \P_{x_0}\bigl[X_1=y_1,\dots,X_{n_1}=y_{n_1}\bigr] \nonumber\\
&& \quad \cdot \sum_{n_2\geq 1} \sum_{(y_{n_1},z_1,\dots,z_{n_2})\in \mathcal{P}^{(2)}_{2,y_{n_1},n_2,l_2}} \P_{y_{n_1}}\bigl[X_1=z_1,\dots,X_{n_2}=z_{n_2}\bigr] \nonumber\\
&&\quad \cdot \P_{z_{n_2}}\bigl[\forall l\geq 1: X_l \in C(z_{n_2})\bigr].\label{equ:L1-L2}
\end{eqnarray}
By case distinction on the different values of $\e{\tau_1}$ and $X_{\e{\tau_1}}$, we get
\begin{eqnarray}
1 &=& \P[\e{\tau_1}<\infty ]\nonumber\\
&=& \sum_{\substack{z_0\in V:\\ \P[X_{\e{\tau_{1}}}=z_0]>0}} \sum_{m_1\geq 1} \sum_{(o,w_1,\dots,w_{m_1})\in \mathcal{P}^{(1)}_{1,z_0,m_1}} \P\left[ \begin{array}{c} X_1=w_1,\dots,X_{m_1}=w_{m_1},\\ \forall l>m: X_l\in C(z_0)\end{array}\right] \label{equ:e-tau1<infty}\\
&=& \sum_{\substack{z_0\in V:\\ \P[X_{\e{\tau_{1}}}=z_0]>0}} \sum_{m_1\geq 1} \sum_{(o,w_1,\dots,w_{m_1})\in \mathcal{P}^{(1)}_{1,z_0,m_1}} \P\bigl[X_1=w_1,\dots,X_{m_1}=w_{m_1}\bigr] \cdot (1-\xi_1).\nonumber
\end{eqnarray}
Observe that, for $z_0\in V$ with $\P[X_{\e{\tau_{1}}}=z_0]>0$, the mapping
$$
\mathcal{P}^{(2)}_{2,y_{n_1},n_2,l_2} \ni (y_{n_1},z_1,\dots,z_{n_2}) \mapsto
(z_0,z_1',\dots,z_{n_2}')\in \mathcal{P}^{(2)}_{2,z_0,n_2,l_2},
$$
where $z_i':=z_0(y_{n_1}^{-1}z_i)$ for $i\in\{1,\dots,n_2\}$, is measure-preserving (Lemma \ref{lem:prop-invariance}), that is, 
\begin{equation}\label{equ:path-shift}
\P_{y_{n_1}}[X_1=z_1,\dots,X_{n_2}=z_{n_2}]=\P_{z_0}[X_1=z_1',\dots,X_{n_2}=z_{n_2}'].
\end{equation}
Furthermore, we remark that
\begin{eqnarray*}
\P[\L_1=l_1] &=& \sum_{\substack{x_0\in V:\\ \P[X_{\e{\tau_{0}}}=x_0]>0}} \sum_{m\geq 1} \sum_{(o,w_1,\dots,w_m)\in \mathcal{P}^{(1)}_{0,x_0,m}} \P\bigl[X_1=w_1,\dots,X_m=w_m\bigr] \\
&&\quad \cdot \sum_{n_1\geq 1} \sum_{(x_0,y_1,\dots,y_{n_1})\in \mathcal{P}^{(2)}_{1,x_0,n_1,l_1}} \P_{x_0}\bigl[X_1=y_1,\dots,X_{n_1}=y_{n_1}\bigr] \cdot (1-\xi_1), \\
\P[\L_2=l_2] &=& \sum_{\substack{z_0\in V:\\ \P[X_{\e{\tau_{1}}}=z_0]>0}} \sum_{m_1\geq 1} \sum_{(o,w_1,\dots,w_{m_1})\in \mathcal{P}^{(1)}_{1,z_0,m_1}} \P\bigl[X_1=w_1,\dots,X_{m_1}=w_{m_1}\bigr]  \\
&&\quad \cdot \sum_{n_2\geq 1} \sum_{(z_0,z_1,\dots,z_{n_2})\in \mathcal{P}^{(2)}_{2,z_0,n_2,l_2}} \P_{z_0}\bigl[X_1=z_1,\dots,X_{n_2}=z_{n_2}\bigr] \cdot (1-\xi_1).
\end{eqnarray*}
We use Equation (\ref{equ:path-shift}) together with  (\ref{equ:e-tau1<infty}) in order to obtain the required independence equation from (\ref{equ:L1-L2}):
\pagebreak[4]
\begin{eqnarray*}
&&\P[\L_1=l_1,\L_2=l_2] \\
&=& \sum_{\substack{x_0\in V:\\ \P[X_{\e{\tau_{0}}}=x_0]>0}} \sum_{m\geq 1} \sum_{(o,w_1,\dots,w_m)\in \mathcal{P}^{(1)}_{0,x_0,m}} \P\bigl[X_1=w_1,\dots,X_m=w_m\bigr] \\
&&\quad \cdot \sum_{n_1\geq 1} \sum_{(x_0,y_1,\dots,y_{n_1})\in \mathcal{P}^{(2)}_{1,x_0,n_1,l_1}} \P_{x_0}\bigl[X_1=y_1,\dots,X_{n_1}=y_{n_1}\bigr] \\
&&\quad \cdot \underbrace{\sum_{\substack{z_0\in V:\\ \P[X_{\e{\tau_{1}}}=z_0]>0}} \sum_{m_1\geq 1} \sum_{(o,w_1,\dots,w_{m_1})\in \mathcal{P}^{(1)}_{1,z_0,m_1}} \P\bigl[X_1=w_1,\dots,X_{m_1}=w_{m_1}\bigr] \cdot (1-\xi_1)}_{=1} \\
&&\quad \cdot \sum_{n_2\geq 1} \sum_{(y_{n_1},z_1,\dots,z_{n_2})\in \mathcal{P}^{(2)}_{2,y_{n_1},n_2,l_2}} \P_{y_{n_1}}\bigl[X_1=z_1,\dots,X_{n_2}=z_{n_2}\bigr] \cdot (1-\xi_1)\\
&\stackrel{(\ref{equ:path-shift})}{=}& \biggl(\sum_{\substack{x_0\in V:\\ \P[X_{\e{\tau_{0}}}=x_0]>0}} \sum_{m\geq 1} \sum_{(o,w_1,\dots,w_m)\in \mathcal{P}^{(1)}_{0,x_0,m}} \P\bigl[X_1=w_1,\dots,X_m=w_m\bigr] \\
&&\quad \cdot \sum_{n_1\geq 1} \sum_{(x_0,y_1,\dots,y_{n_1})\in \mathcal{P}^{(2)}_{1,x_0,n_1,l_1}} \P_{x_0}\bigl[X_1=y_1,\dots,X_{n_1}=y_{n_1}\bigr] \cdot (1-\xi_1)\biggr)\\
&&\quad \cdot \biggl( \sum_{\substack{z_0\in V:\\ \P[X_{\e{\tau_{1}}}=z_0]>0}} \sum_{m_1\geq 1} \sum_{(o,w_1,\dots,w_{m_1})\in \mathcal{P}^{(1)}_{1,z_0,m_1}} \P\bigl[X_1=w_1,\dots,X_{m_1}=w_{m_1}\bigr]  \\
&&\quad \cdot \sum_{n_2\geq 1} \sum_{(z_0,z_1,\dots,z_{n_2})\in \mathcal{P}^{(2)}_{2,z_0,n_2,l_2}} \P_{z_0}\bigl[X_1=z_1,\dots,X_{n_2}=z_{n_2}\bigr]\cdot (1-\xi_1)\biggr) \\
&=& \P[L_1=l_1]\cdot \P[L_2=l_2].
\end{eqnarray*} 
The proof of independence of $(T_{i}-T_{i-1})_{i\in\N}$ works completely analogously: for $n_1,n_2\in\N$, we can rewrite
\begin{eqnarray*}
&&\P[T_1-T_0=n_1,T_2-T_1=n_2] \\
&=& \sum_{\substack{x_0\in V:\\ \P[X_{\e{\tau_{0}}}=x_0]>0}} \sum_{m\geq 1} \sum_{(o,w_1,\dots,w_m)\in \mathcal{P}^{(1)}_{0,x_0,m}} \P\bigl[X_1=w_1,\dots,X_m=w_m\bigr] \\
&&\quad \cdot  \sum_{(x_0,y_1,\dots,y_{n_1})\in \bigcup_{l_1\in\mathrm{supp}(\L_1)}\mathcal{P}^{(2)}_{1,x_0,n_1,l_1}} \P_{x_0}\bigl[X_1=y_1,\dots,X_{n-1}=y_{n_1}\bigr] \\
&& \quad \cdot  \sum_{(y_{n_1},z_1,\dots,z_{n_2})\in \bigcup_{l_2\in\mathrm{supp}(\L_1)} \mathcal{P}^{(2)}_{2,y_{n_1},n_2,l_2}} \P_{y_{n_1}}\bigl[X_1=z_1,\dots,X_{n_2}=z_{n_2}\bigr]  \\
&&\quad \cdot \P_{z_{n_2}}\bigl[\forall l\geq 1: X_l \in C(z_{n_2})\bigr] 
\end{eqnarray*}
\begin{eqnarray*}
&\stackrel{(\ref{equ:path-shift})}{=}& \sum_{\substack{x_0\in V:\\ \P[X_{\e{\tau_{0}}}=x_0]>0}} \sum_{m\geq 1} \sum_{(o,w_1,\dots,w_m)\in \mathcal{P}^{(1)}_{0,x_0,m}} \P\bigl[X_1=w_1,\dots,X_m=w_m\bigr] \\
&&\quad \cdot \sum_{(x_0,y_1,\dots,y_{n_1})\in \bigcup_{l_1\in\mathrm{supp}(\L_1)}\mathcal{P}^{(2)}_{1,x_0,n_1,l_1}} \P_{x_0}\bigl[X_1=y_1,\dots,X_{n_1}=y_{n_1}\bigr] \cdot (1-\xi_1)\\
&&\quad \cdot \underbrace{\sum_{\substack{z_0\in V:\\ \P[X_{\e{\tau_{1}}}=z_0]>0}} \sum_{m_1\geq 1} \sum_{(o,w_1,\dots,w_{m_1})\in \mathcal{P}^{(1)}_{1,z_0,m_1}} \P\bigl[X_1=w_1,\dots,X_{m_1}=w_{m_1}\bigr] \cdot (1-\xi_1)}_{=1} \\
&&\quad \cdot  \sum_{(z_0,z_1,\dots,z_{n_2})\in \bigcup_{l_2\in\mathrm{supp}(\L_1)}\mathcal{P}^{(2)}_{2,z_0,n_2,l_2}} \P_{z_0}\bigl[X_1=z_1,\dots,X_{n_2}=z_{n_2}\bigr] \\
&=& \P[T_1-T_0=n_1]\cdot \P[T_2-T_1=n_2],
\end{eqnarray*} 
since
\begin{eqnarray*}
\P[T_1-T_0=n_1] &=& \sum_{\substack{x_0\in V:\\ \P[X_{\e{\tau_{0}}}=x_0]>0}} \sum_{m\geq 1} \sum_{(o,w_1,\dots,w_m)\in \mathcal{P}^{(1)}_{0,x_0,m}} \P\bigl[X_1=w_1,\dots,X_m=w_m\bigr] \\
&&\quad \cdot \sum_{(x_0,y_1,\dots,y_{n_1})\in \bigcup_{l_1\in\mathrm{supp}(\L_1)} \mathcal{P}^{(2)}_{1,x_0,n_1,l_1}} \P_{x_0}\bigl[X_1=y_1,\dots,X_{n_1}=y_{n_1}\bigr] \cdot (1-\xi_1) \\
\P[T_2-T_1=n_2] &=& \sum_{\substack{z_0\in V:\\ \P[X_{\e{\tau_{1}}}=z_0]>0}} \sum_{m_1\geq 1} \sum_{(o,w_1,\dots,w_{m_1})\in \mathcal{P}^{(1)}_{1,z_0,m_1}} \P\bigl[X_1=w_1,\dots,X_{m_1}=w_{m_1}\bigr],  \\
&&\quad \cdot \sum_{(z_0,z_1,\dots,z_{n_2})\in \bigcup_{l_2\in\mathrm{supp}(\L_1)} \mathcal{P}^{(2)}_{2,z_0,n_2,l_2}} \P_{z_0}\bigl[X_1=z_1,\dots,X_{n_2}=z_{n_2}\bigr]\cdot (1-\xi_1).
\end{eqnarray*}
Once again, the general proof for independence of  $(T_{i}-T_{i-1})_{i\in\N}$ works completely analogously, and we leave it as an exercise to the interested reader. This finishes the proof of the proposition.
\end{proof}
For $n\in\N$, set 
$$
\mathbf{t}(n):=\max\{m\in\N_0 \mid T_m\leq n\}.
$$ 
Since $(\W{k},\psi_k)_{k\in\N}$ is positive-recurrent, a geometric standard argument implies that \mbox{$\mathbf{t}(n)\to\infty$} almost surely.
From Propositions \ref{lem:increment-bound} and \ref{prop:iid-sequence} follows that
$$
\frac{1}{\mathbf{t}(n)}\sum_{j=1}^{\mathbf{t}(n)} T_j-T_{j-1}\xrightarrow{n\to\infty}\mathbb{E}[T_1-T_0] \ \textrm{almost surely},
$$
which in turn yields with Lemma \ref{lem:expT0-finite} that
\begin{equation}\label{equ:e-tau-convergence}
\frac{\e{\tau_{\mathbf{t}(n)}}}{\mathbf{t}(n)}=\frac{T_{\mathbf{t}(n)}}{\mathbf{t}(n)}\xrightarrow{n\to\infty} \mathbb{E}[T_1-T_0] \ \textrm{almost surely.}
\end{equation}
This observation will be helpful in the following corollary, for which we recall the definition of $\L_1=\widetilde{\mathbf{L}}_1-(T_1-T_0)\mathfrak{r}$:

\begin{Cor}\label{cor:exp=0}
$$
\mathbb{E}[\L_1]=0 \quad \textrm{ and }\quad \sigma_{\L}^2=\mathbb{E}\bigl[\bigl((\mathbf{R}_{T_1}-\mathbf{R}_{T_0})- (T_1-T_0)\cdot \mathfrak{r}\bigr)^2\bigr]>0.
$$
\end{Cor}
\begin{proof}
Analogously to Theorem \ref{thm:range-existence} and and (\ref{equ:R-convergence}) (replace the exit times $\mathbf{e}_i$ by $T_i$) one can prove that
\begin{equation}\label{equ:T-formula}
\mathfrak{r}=\lim_{n\to\infty} \frac{\mathbf{R}_{\mathbf{e}_{\tau_{\mathbf{t}(n)}}}}{\mathbf{t}(n)}\frac{\mathbf{t}(n)}{\mathbf{e}_{\tau_{\mathbf{t}(n)}}}
=
\frac{\mathbb{E}\bigl[\mathbf{R}_{T_1}-\mathbf{R}_{T_0}\bigr]}{\mathbb{E}[T_1-T_0]}=\frac{\mathbb{E}[\widetilde{\L}_1]}{\mathbb{E}[T_1-T_0]}.
\end{equation}
Therefore, $\mathbb{E}[\L_1]=0$ and the proposed formula for $\mathrm{Var}(\L_1)$ follows. 
\par
It remains to show that $(\mathbf{R}_{T_1}-\mathbf{R}_{T_0})- (T_1-T_0)\cdot \mathfrak{r}$ is \textit{not} almost surely constant. To this end we construct two  paths (having positive probability to be realised) of different length but which visit the same vertices.
Take any $x_0\in V_i$, $i\in\calI$, with $p^{(n_x)}(x_0,x_0)>0$ for some $n_x\in \N$; this choice is possible because of the assumption (A) made at the  beginning of Subsection \ref{subsec:free-products}. If $x_0\in V_2$ (otherwise, swap the roles of $V_1$ and $V_2$), take now any path inside $C(g_0)$ from $g_0$ to $g_0x_0g_0$ which visits $g_0x_0$ twice, say 
$$ 
(g_0,g_1,\dots,g_{j-1},g_0x_0,g_{j+1},\dots,g_{j+k-1},g_0x_0,g_{j+k+1},\dots,g_0x_0g_0).
$$
We can add another loop at $g_0x_0$ as follows:
$$
(g_0,g_1,\dots,g_{j-1},g_0x_0,g_{j+1},\dots,g_{j+k-1},g_0x_0,g_{j+1},\dots,g_{j+k-1},g_0x_0,g_{j+k+1},\dots,g_0x_0g_0).
$$
Both paths visit the same elements of $V$, but have different lengths. That is, there are $z_1,z_2\in\mathbb{R}$, $z_1\neq z_2$, such that $\P[\L_1=z_1],\P[\L_1=z_2]>0$, providing $\sigma_{\L}^2>0$.
\end{proof}

We introduce further notation: for $k\in\mathbb{N}$,  set
\[
\mathbf{S}_k:= \sum_{i=1}^k \L_i, \quad \widetilde{\mathbf{S}}_k:= \sum_{i=1}^k \widetilde{\L}_i.
\]
We will need the following useful convergence behaviour:
\begin{Lemma} \label{lem:overhead}
$$
\frac{\mathbf{R}_n-\widetilde{\mathbf{S}}_{\mathbf{t}(n)}}{\sqrt{n}}\xrightarrow{\P} 0.
$$
\end{Lemma}
\begin{proof}
We have the following upper bound for the nominator above:
$$
0\leq \mathbf{R}_n-\widetilde{\mathbf{S}}_{\mathbf{t}(n)} \leq \mathbf{R}_{T_{\mathbf{t}(n)+1}} - \widetilde{\mathbf{S}}_{\mathbf{t}(n)} \leq T_{\mathbf{t}(n)+1}-T_{\mathbf{t}(n)}+T_0.
$$ 
Since $(T_i-T_{i-1})_{i\in\mathbb{N}}$ is an i.i.d. sequence, we obtain for  $\varepsilon >0$ and $n$ large enough:
\begin{eqnarray*}
&&\mathbb{P}\bigl[\mathbf{R}_n-\widetilde{\mathbf{S}}_{\mathbf{t}(n)}  > \varepsilon \sqrt{n}, \mathbf{t}(n)\geq 1\bigr] \\
&\leq & \mathbb{P}\bigl[T_{\mathbf{t}(n)+1}-T_{\mathbf{t}(n)}+T_0 > \varepsilon \sqrt{n}, \mathbf{t}(n)\geq 1\bigr] \\
&\leq & \mathbb{P}\bigl[\exists k\in\{1,\dots,n\}: T_{k+1}-T_{k}+T_0 > \varepsilon \sqrt{n}\bigr] \\
&\leq & \mathbb{P}\Bigl[\exists k\in\{1,\dots,n\}: T_{k+1}-T_{k} > \frac{\varepsilon}{2} \sqrt{n}\Bigr] +\mathbb{P}\Bigl[T_0 >  \frac{\varepsilon}{2} \sqrt{n}\Bigr]\\
&\leq & n\cdot \mathbb{P}\Bigl[T_{1}-T_{0} > \frac{\varepsilon}{2} \sqrt{n}\Bigr] +\mathbb{P}\Bigl[T_0 >  \frac{\varepsilon}{2} \sqrt{n}\Bigr]\\
&\leq & n\cdot \mathbb{P}\Bigl[(T_{1}-T_{0})^4 > \frac{\varepsilon^4}{2^4} n^2\Bigr] +\mathbb{P}\Bigl[T_0 >  \frac{\varepsilon}{2} \sqrt{n}\Bigr]\\
&\leq & n\cdot \frac{\mathbb{E}\bigl[(T_1-T_0)^4\bigr]}{\frac{\varepsilon^4}{2^4}n^2} + \frac{\mathbb{E}[T_0]}{\frac{\varepsilon}{2} \sqrt{n}} \xrightarrow{n\to\infty} 0.
\end{eqnarray*}
In the last inequality we applied Markov's Inequality and used Proposition \ref{lem:increment-bound}.
As $\mathbf{t}(n)\to \infty$ almost surely, we have established the proposed convergence.
\end{proof}

\begin{proof}[Proof of Theorem \ref{thm:clt}]
First, we remark that $\sigma^2$ as defined in Theorem \ref{thm:clt} is strictly positive due to Corollary \ref{cor:exp=0}.
By Billingsley \cite[Theorem 14.4]{billingsley:99} together with Corollary \ref{cor:exp=0}, we obtain the following convergence in law:
$$
\frac{\mathbf{S}_{\mathbf{t}(n)}}{\sigma_{\L} \sqrt{\mathbf{t}(n)}} \xrightarrow{\mathcal{D}} N(0,1).
$$
The convergence in (\ref{equ:e-tau-convergence}) implies
$$
0\leq \frac{n-\e{\tau_{\mathbf{t}(n)}}}{\mathbf{t}(n)} \leq \frac{\e{\tau_{\mathbf{t}(n)+1}}-\e{\tau_{\mathbf{t}(n)}}}{\mathbf{t}(n)}
\xrightarrow{n\to\infty} 0 \ \textrm{almost surely},
$$
from which we obtain 
$$
\frac{n}{\mathbf{t}(n)} = \frac{n-\e{\tau_{\mathbf{t}(n)}}}{\mathbf{t}(n)}+ \frac{\e{\tau_{\mathbf{t}(n)}}}{\mathbf{t}(n)} \xrightarrow{n\to\infty} \mathbb{E}[T_1-T_0] \ \textrm{almost surely.}
$$
An application of the Lemma of Slutsky gives:
\begin{equation}\label{equ:clt}
\frac{\mathbf{S}_{\mathbf{t}(n)}}{ \sqrt{n}} = \frac{\mathbf{S}_{\mathbf{t}(n)}}{\sigma_{\L} \sqrt{\mathbf{t}(n)}}\frac{\sqrt{\mathbf{t}(n)}}{\sqrt{n}}\sigma_{\L}\xrightarrow{\mathcal{D}} N(0,\sigma^2),
\end{equation}
where $\sigma^2$ is given as stated in Theorem \ref{thm:clt}.
\par
Observe now that, for $n\geq T_0$,
$$
\mathbf{S}_{\mathbf{t}(n)}=(\mathbf{R}_{T_{\mathbf{t}(n)}}-\mathbf{R}_{T_0})-(T_{\mathbf{t}(n)}-T_0)\cdot \mathfrak{r}.
$$
For $\varepsilon >0$, we get:
\begin{eqnarray}
&&\P\bigl[ \bigl| \mathbf{S}_{\mathbf{t}(n)} -(\mathbf{R}_n-n\cdot\mathfrak{r})\bigr| >\varepsilon \sqrt{n}\bigr]\nonumber\\
&\leq & \P\Bigl[ \mathbf{R}_n-\mathbf{R}_{T_{\mathbf{t}(n)}}+\mathbf{R}_{T_0} \geq  \frac{\varepsilon}{2} \sqrt{n}\Bigr] +
\P\Bigl[ \mathfrak{r}\cdot \bigl(n - (T_{\mathbf{t}(n)}-T_0)\bigr) \geq  \frac{\varepsilon}{2} \sqrt{n}\Bigr]\label{equ:inequ1}.
\end{eqnarray}
From Lemma \ref{lem:overhead} follows that
$$
\frac{\mathbf{R}_n-\mathbf{R}_{T_{\mathbf{t}(n)}}+\mathbf{R}_{T_0}}{\sqrt{n}} = \frac{\mathbf{R}_n-\sum_{j=1}^{\tau_{\mathbf{t}(n)}} \tR_j}{\sqrt{n}}
=\frac{\mathbf{R}_n-\widetilde{\mathbf{S}}_{\mathbf{t}(n)}}{\sqrt{n}}\xrightarrow{\P} 0,
$$
hence $\P\Bigl[ \mathbf{R}_n-\mathbf{R}_{T_{\mathbf{t}(n)}}+\mathbf{R}_{T_0} \geq  \frac{\varepsilon}{2} \sqrt{n}\Bigr]\to 0$ as $n\to\infty$.
For the second summand in (\ref{equ:inequ1}), we obtain
\begin{eqnarray*}
&& \P\Bigl[ \mathfrak{r}\cdot \bigl(n - (T_{\mathbf{t}(n)}-T_0)\bigr) \geq  \frac{\varepsilon}{2} \sqrt{n},\mathbf{t}(n)\geq 1\Bigr] \\
&\leq & \P\Bigl[ \mathfrak{r}\cdot \bigl(T_{\mathbf{t}(n)+1} - (T_{\mathbf{t}(n)}-T_0)\bigr) \geq  \frac{\varepsilon}{2} \sqrt{n},\mathbf{t}(n)\geq 1\Bigr]\\
&\leq & \P\Bigl[ \exists k\in\{1,\dots,n\}: T_{k+1}-(T_k-T_0) \geq  \frac{\varepsilon}{2\mathfrak{r}} \sqrt{n}\Bigr] \\
&\leq  & \P\Bigl[\exists k\in\{1,\dots,n\}: T_{k+1}-T_k \geq  \frac{\varepsilon}{4\mathfrak{r}} \sqrt{n}\Bigr]+\P\Bigl[T_{0} \geq  \frac{\varepsilon}{4\mathfrak{r}} \sqrt{n}\Bigr]\\
&\stackrel{\textrm{Prop. \ref{prop:iid-sequence}}}{\leq} & n\cdot \P\Bigl[T_{1}-T_0 \geq  \frac{\varepsilon}{4\mathfrak{r}} \sqrt{n}\Bigr] +  \P\Bigl[T_0 \geq  \frac{\varepsilon}{4\mathfrak{r}} \sqrt{n}\Bigr] \\ 
&= & n\cdot \P\Bigl[(T_{1}-T_0)^4 \geq  \frac{\varepsilon^4}{(4\mathfrak{r})^4} n^2\Bigr] +  \P\Bigl[T_0 \geq  \frac{\varepsilon}{4\mathfrak{r}} \sqrt{n}\Bigr] \\ 
&\leq &  (4\mathfrak{r})^4\cdot n\cdot \frac{\mathbb{E}\bigl[(T_1-T_0)^4\bigr]}{\varepsilon^4 n^2}+
4\mathfrak{r} \cdot \frac{\mathbb{E}\bigl[T_0\bigr]}{\varepsilon \sqrt{n}} \xrightarrow{n\to\infty} 0.
\end{eqnarray*}
We applied Markov's Inequality in the last line together with Proposition \ref{lem:increment-bound} and Lemma \ref{lem:expT0-finite}.
As $\mathbf{t}(n)\to\infty$ almost surely for $n\to\infty$, we obtain
$$
\P\bigl[ \bigl| S_{\mathbf{t}(n)} -(\mathbf{R}_n-n\cdot\mathfrak{r})\bigr| >\varepsilon \sqrt{n}\bigr] \xrightarrow{\P} 0.
$$
%
%
Another application of the Lemma of Slutsky together with (\ref{equ:clt}) proves the claim.
\end{proof}

\section{Analyticity of the Asymptotic Range}
\label{sec:analyticity}

In this section we prove Theorem \ref{thm:analyticity} which states that $\mathfrak{r}$ varies real-analytically in terms of probability measures of constant support depending on finitely many parameters. First, we describe the problem more formally. From now on we \textit{assume} that the transition probabilities depend on finitely many parameters $p_1,\dots, p_d$, $d\in\N$, taking values in $(0,1)$; that is, if $p(x,y)>0$, $x,y\in V_i$, then $p(x,y)=p_j$ for some $j\in\{1,\dots,d\}$. 
\par
More precisely, we write $E_i:=\{(x_i,y_i)\in V_i^2 \mid x_i \to y_i\}$ for  the set of edges of $\mathcal{X}_i$, $i\in\{1,2\}$. Let $\eta: E_1\cup E_2 \to\{p_1,\dots,p_d\}$ be a mapping. We say that a parameter vector \mbox{$\underline{p}:=(p_1,\dots,p_d)\in (0,1)^{d}$} permits a well-defined random walk on $V$ if
\begin{eqnarray*}
&&\forall x_1\in V_1: \sum_{y_1\in V_1: (x_1,y_1)\in E_1} \eta(x_1,y_1)+\sum_{y_2\in V_2: (o_2,y_2)\in E_2} \eta(o_2,y_2)=1\quad \textrm{ and }\\
&&\forall x_2\in V_2: \sum_{y_2\in V_2: (x_2,y_2)\in E_2} \eta(x_2,y_2)+\sum_{y_1\in V_1: (o_1,y_1)\in E_1} \eta(o_1,y_1)=1.
\end{eqnarray*}
In this case we set $\alpha:= \sum_{y\in V_1} \eta(o_1,y)$,  
$p_1(x_1,y_1):=\eta(x_1,y_1)/\alpha$ and $p_2(x_2,y_2):=\eta(x_2,y_2)/(1-\alpha)$ for $x_1,y_1\in V_1$, $x_2,y_2\in V_2$, which defines a well-defined random walk on the free product $V$.

We write
$$
\mathcal{P}:=\Bigl\lbrace \underline{p}:=(p_1,\dots,p_d)\in (0,1)^{d} \,\Bigl|\, \underline{p} \textrm{ defines a well-defined random walk on $V$}\Bigr\rbrace,
$$ 
the set of parameter vectors which allow well-defined random walks.
Our aim is to show that the mapping
$$
\mathcal{P} \ni \underline{p}\mapsto \mathfrak{r}=\mathfrak{r}(\underline{p})
$$
varies real analytically in $\underline{p}=(p_1,\dots,p_d)\in\mathcal{P}$, that is, $\mathfrak{r}(\underline{p})$ can be expanded as a multivariate power series in the variables $p_1,\dots,p_d$  in a neighbourhood of any $\underline{p}_0\in\mathcal{P}$. 
\par
We have to give some preliminary remarks, before we are able to prove the proposed result. Let $A_n\in \sigma(X_0,X_1,\dots,X_n)$, $n\in\N_0$, be an event which can be described by paths of length $n$ of the Markov chain $(X_n)_{n\in\N_0}$ on $V$; e.g., $A_n=[X_n\in V_1]$. By decomposing each such path belonging to $A_n$ according to the number of steps which are performed w.r.t. the parameters $p_1,\dots,p_d$, we can rewrite $\P(A_n)$ as 
\begin{equation}
 \sum_{\substack{n_1,\dots,n_{d}\geq 0:\\ n_1+\dots+n_{d}= n}} c(n_1,\dots,n_{d})\cdot p_1^{n_1}\cdot
\ldots \cdot p_d^{n_d}, \label{equ:path-prob}
\end{equation}
where $c(n_1,\dots,n_{d})\in\N_0$.
If the generating function $\mathcal{F}(z):= \sum_{n\geq 0}\P(A_n)\,z^n$, $z\in\mathbb{C}$, has radius of convergence strictly bigger than $1$, then, for $\delta>0$ small enough,
\begin{eqnarray}
\infty >\mathcal{F}(1+\delta)&=&\sum_{n\geq 0} \sum_{\substack{n_1,\dots,n_{d}\geq 0:\\ n_1+\dots+n_{d}=n}} c(n_1,\dots,n_{d})\bigl( p_1(1+\delta)\bigr)^{n_1}\cdot\ldots\cdot\bigl( p_d  (1+\delta)\bigr)^{n_d};\label{equ:analytic1}
\end{eqnarray}
that is, the mapping $\underline{p}\mapsto \mathcal{F}(1)$ varies real-analytically in a neighbourhood of any $\underline{p}_0=(p_1,\dots,p_d)\in\mathcal{P}$ when considered as a power series in $\underline{p}$. 
\par
We apply this observation to the Green functions $G(x,y|z)$, $x,y\in V$: since $G(x,y|z)$ has radius of convergence strictly bigger than $1$ and $p^{(n)}(x,y)=\P_x[X_n=y]$ can be rewritten as in (\ref{equ:path-prob}), there is some neighbourhood $\mathfrak{U}$ of $\underline{p}_0$ in $\mathcal{P}\subseteq \mathbb{R}^{d}$ such that the Green functions w.r.t. $\underline{p}\in\mathfrak{U}$ have still radius of convergence strictly bigger than $1$. Hence, $\mathfrak{r}(\underline{p})$ exists for every $\underline{p}\in \mathfrak{U}$.
\par
We apply the observation above also to  the formula for $\mathfrak{r}$ given in (\ref{equ:T-formula}):
$$
\mathfrak{r}=\frac{\mathbb{E}[\widetilde{\L}_1]}{\mathbb{E}[T_1-T_0]}.
$$
We now show that both nominator and denominator vary real-analytically in $\underline{p}\in\mathcal{P}$.

\begin{Lemma}\label{lem:T1-T0-analytic}
The mapping $ \mathcal{P}\ni \underline{p} \mapsto \mathbb{E}[T_1-T_0]$ varies real-analytically.
\end{Lemma}
\begin{proof}
First, observe that we can rewrite the expectation as
$$
\mathbb{E}[T_1-T_0]=\sum_{n\geq 1}\mathbb{P}[T_1-T_0=n]\cdot n = \frac{\partial}{\partial z}\biggl[\sum_{n\geq 1}\mathbb{P}[T_1-T_0=n]\cdot z^n\biggr]\Biggl|_{z=1}.
$$
Recall from
Proposition \ref{lem:increment-bound} that the power series $\sum_{n\geq 1}\mathbb{P}[T_1-T_0=n]\cdot z^n$ has radius of convergence strictly bigger than $1$.  According to the remarks at the beginning of this section it suffices to show that the probabilities $\mathbb{P}[T_1-T_0=n], n\in\N$ can be written in the form of (\ref{equ:path-prob}). We decompose this probability according to the values of $T_0$, $T_1-T_0$, $X_{T_0}$ and $X_{T_1}$:
\begin{eqnarray}
&& \P[T_1-T_0=n] \nonumber\\
&=& \sum_{m\in\N}\sum_{w\in V_{g_0}}\sum_{y\in V_{g_0}^{(2)}}  \P\bigl[ X_{T_0}=w,T_0=m,X_{T_1}=wy, T_1=m+n\bigr] \nonumber\\
&=&  \sum_{m\in\N} \sum_{w\in V_{g_0}} \P\bigl[ X_m=w,\forall m'<m: X_{m'}\neq w\bigr] \nonumber  \\
&& \quad \cdot \sum_{y\in V_{g_0}^{(2)}}  \P_w\bigl[ X_n=wy,\forall n'\leq n: X_{n'}\in C(w) \wedge X_{n'}\notin C(wy)\bigr]  \cdot (1-\xi_1) \nonumber \\
&=& \sum_{y\in V_{g_0}^{(2)}}  \P\bigl[ X_n=y,\forall n'\leq n: X_{n'}\notin V_1^\times \cup C(y) \bigr].\label{equ:prob-T1-T0}
\end{eqnarray}
We used the equation
\begin{equation}\label{equ:T0-finite}
\sum_{m\in\N,w\in V_{g_0}} \mathbb{P}\bigl[X_m=w,\forall m'<m: X_{m'}\neq w\bigr]\cdot (1-\xi_1)=\mathbb{P}[T_0<\infty]=1
\end{equation}
and the measure preserving shift of paths in $C(w), w\in V_{g_0}$ to paths in $V\setminus V_1^\times$ by the shift transformation $g\mapsto w^{-1}g$.
Since every single probability in the sum (\ref{equ:prob-T1-T0}) can be rewritten in the form of (\ref{equ:path-prob}), we finally get  analyticity of $\mathbb{E}[T_1-T_0]$ as explained in (\ref{equ:analytic1}).
\end{proof}

\begin{Lemma}\label{lem:EL-analytic}
The mapping $\mathcal{P}\ni \underline{p} \mapsto \mathbb{E}[\widetilde{\L}_1]$ varies real-analytically.
\end{Lemma}
\begin{proof}
Since $0\leq \widetilde{\L}_1\leq T_1-T_0$, we can rewrite the expectation as
\begin{eqnarray}
\mathbb{E}[\widetilde{\L}_1]&=& \sum_{m\geq 1}m\cdot \mathbb{P}[\widetilde{\L}_1=m] =
\sum_{m\geq 1}\sum_{n\geq m} m\cdot \mathbb{P}[\widetilde{\L}_1=m,T_1-T_0=n]\nonumber \\
&=& \sum_{n\geq 1}\underbrace{\sum_{m=1}^n m\cdot  \mathbb{P}[\widetilde{\L}_1=m,T_1-T_0=n]}_{a_n=:}\cdot z^n\Biggl|_{z=1}. \nonumber 
\end{eqnarray}
For real $z>0$, we have
$$
\sum_{n\geq 1} a_n\cdot z^n \leq z\cdot \frac{\partial}{\partial z}\biggl[\sum_{n\geq 1}\mathbb{P}[T_1-T_0=n]\cdot z^n\biggr],
$$
which yields together with Proposition \ref{lem:increment-bound}  that the power series $\sum_{n\geq 1}a_n\cdot z^n$ has radius of convergence strictly bigger than $1$.
 According to the remarks at the beginning of this section it suffices to show that the probabilities $\mathbb{P}[\widetilde{\L}_1=m,T_1-T_0=n], m,n\in\N$ can be written in the form of (\ref{equ:path-prob}). To this end, we introduce further notation: for $w\in V_{g_0}, n,m\in\N$, denote by  $\mathcal{P}^{(3)}_{w,m,n}$ the set of paths $(w,y_1,\dots,y_n)\in V^{n+1}$ of length $n$ such that
\[
\bigcup_{t\in\N} \bigl[X_t=w,X_{t+1}=y_1,\dots,X_{t+n}=y_n\bigr] \cap \bigl[X_{t}=w,T_0=t,T_1-T_0=n,\widetilde{\L}_1=m\bigr]\neq \emptyset.
\]
Then we get the required form of $\mathbb{P}[\widetilde{\L}_1=m,T_1-T_0=n]$ by decomposition according to the values of $T_0$, $X_{T_0}$ and $X_{T_1}$:
\begin{eqnarray*}
&&\P[\widetilde{\L}_1=m,T_1-T_0=n] \\
&=& \sum_{w\in V_{g_0}}\sum_{m_0\in\N} \sum_{y\in V^{(2)}_{g_0}} \mathbb{P}\bigl[X_{m_0}=w,T_0=m_0,X_{T_1}=wy,\widetilde{\L}_1=m,T_1-T_0=n \bigr]\\
&=& \sum_{w\in V_{g_0}} \sum_{m_0\in\mathbb{N}} \P\bigl[X_{m_0}=w,\forall m'<m_0: X_{m'}\notin C(w)\bigr] \\
&&\quad \cdot  \sum_{(w,y_1,\dots,y_{n})\in \mathcal{P}^{(3)}_{w,m,n}} \P_{w}\bigl[X_1=y_1,\dots,X_{n}=y_{n}\bigr] \cdot (1-\xi_1)\\
&=&  \sum_{(g_0,y_1,\dots,y_{n})\in \mathcal{P}^{(3)}_{g_0,m,n}} \P_{g_0}\bigl[X_1=y_1,\dots,X_{n}=y_{n}\bigr].  
\end{eqnarray*}
In the last equation we used (\ref{equ:T0-finite}) 
and the measure preserving shift of paths in $\mathcal{P}^{(3)}_{w,m,n}$ to paths in $\mathcal{P}^{(3)}_{g_0,m,n}$ by the transformation $V^{(2)}_{g_0}\ni y_i\mapsto g_0(w^{-1}y_i)$. 
   This finishes the proof.
\end{proof}

\begin{proof}[Proof of Theorem \ref{thm:analyticity}]
The proof follows now directly from Lemmas \ref{lem:T1-T0-analytic} and \ref{lem:EL-analytic} in view of formula (\ref{equ:T-formula}).
\end{proof}

\section{Remarks}

\label{sec:remarks}

In Section \ref{subsec:free-products} we excluded the case $p_i(x,x)>0$ for $x\in V_i$, $i\in\calI$. However, this was just a technical assumption in order to simplify the proofs for sake of better readability. We now explain the modifications needed if we allow $p_i(x,x)>0$ for some $x\in V_i$, $i\in\calI$.
\par
For $n\in\N$, we denote by $\mathfrak{i}_n$ the index $j\in\{1,2\}$ such that the $n$-th step of $(X_n)_{n\in\N_0}$ is performed according to $\overline{P}_j$. For $i\in\calI$, define the stopping time 
$$
\mathbf{s}^{(i)}:=\inf\{m\in\N \mid \mathfrak{i}_m=i, X_{m-1},X_m\in V_i\}. 
$$
For $z\in\N$, set
$$
\zeta_i(z):=\sum_{n\geq 1} \P\bigl[\mathbf{s}^{(i)}=n\bigr]\cdot z^n.
$$
Here, $\P\bigl[\mathbf{s}^{(i)}=n\bigr]$ is the probability that a step within $V_i\subset V$ according
to $\overline{P}_i$ is performed for the first time at time $n$. Recall also the definition
$$
\xi_i:=\sum_{n\geq 1} \P[X_n\in V_i^\times,\forall m<n: X_n\notin V_i^\times],
$$
the probability of visiting any element in $V_i^\times$ after finite time when starting at $o$.  Observe that $\xi_i=\zeta_i(1)$ if $p_j(x,x)=0$ for all $j\in \calI$ and all $x\in V_j$.
\begin{Cor}
For all $i\in\calI$, we have $\xi_i<1$ and $\zeta_i<1$. Moreover, $\zeta_i(z)$, $i\in\calI$, has radius of convergence strictly bigger than $1$.
\end{Cor}
\begin{proof}
In \cite[Lemma 2.3]{gilch:07} it is shown that  $\zeta_i<1$, $\xi_i<1$ respectively, holds for all $i\in\calI$ under the assumption that $p_j(x,x)=0$ for all $j\in \calI$ and all $x\in V_j$. However, the proof works completely analogously without that assumption. Thus, $\zeta_i<1$ and $\xi_i<1$ hold also in the setting where we allow $p_i(x,x)>0$ for $x\in V_i$, $i\in\calI$. Furthermore, $\zeta(z)$ has radius of convergence strictly bigger than $1$, since $G(o,o|z)$ has radius of convergence strictly bigger than $1$ and due to the inequality 
$$
\zeta(z) \leq G(o,o|z)\cdot \alpha_i\cdot z.
$$
\end{proof}
The crucial point is now that Equation (\ref{equ:L-Li}) becomes -- in the general setting --
$$
L(x_i,y_i|z)=L_i\bigl(x_i,y_i\,\bigl|\, \zeta_i(z)\bigr), \quad i\in\calI, x_i,y_i\in V_i.
$$
The proofs of the main results work completely analogously, but one has to replace the functions $\xi_i(z)$ by $\zeta_i(z)$ at some places. However, this does \textit{not} require further additional reasoning, since the involved arguments (e.g., radii of convergence of $\xi_i(z), \zeta_i(z)$ are strictly bigger than $1$) still remain the same. Due to possible nasty case distinctions whether a loop step from any $x\in V$ to $x$ is performed as a step in $V_{\t(x)}$ or not, we excluded loops.

\bibliographystyle{abbrv}
\bibliography{literatur}

\end{document}